\documentclass[svgnames]{article}
\usepackage{geometry}
\usepackage[utf8]{inputenc}
\usepackage{fnpct} 
\usepackage{amsmath}
\usepackage[cal=euler,calscaled=.98,scr=euler,bb=dsserif]{mathalpha}
\usepackage{amssymb} 
\usepackage{mathtools} 
\usepackage{tikz-cd} 
\usepackage{amsthm} 
\usepackage{adjustbox} 
\usepackage{multicol}
\usepackage{textcomp}
\usepackage{xspace}
\usepackage{microtype}
\usepackage[%
colorlinks=true%
,linkcolor=FireBrick%
,urlcolor=SlateBlue%
,citecolor=Teal%
]{hyperref}

\numberwithin{equation}{section}

\theoremstyle{definition} 
\newtheorem{theorem}{Theorem}[section]
\newtheorem{lemma}[theorem]{Lemma}

\newtheorem{definition}[theorem]{Definition}
\newtheorem{example}[theorem]{Example}
\newtheorem{remark}[theorem]{Remark}

\newtheorem*{theorem*}{Theorem} 

\newcommand{\mathsc}[1]{{\normalfont\textsc{#1}}}

\DeclareMathOperator{\cAut}{\mathscr{A}\mspace{-0.7mu}ut}

\DeclareMathOperator{\Tors}{Tors}

\DeclareMathOperator{\Pic}{\mathsc{Pic}}

\newcommand{\moncat}[1]{\mathscr{#1}}
\newcommand{\cA}{\moncat{A}}
\newcommand{\cB}{\moncat{B}}

\newcommand{\cE}{\moncat{E}}
\newcommand{\cL}{\moncat{L}}

\newcommand{\lto}{\longrightarrow}

\newcommand{\loccit}{loc.~cit.\xspace}

\begin{document}

\title{Symmetric Monoidal Bicategories and Biextensions}
\author{%
  Ettore Aldrovandi \\
  {\small \url{ealdrovandi@fsu.edu}}
  \and
  Milind Gunjal \\
  {\small \url{mgunjal@fsu.edu}}
  \and {\small Department of Mathematics, Florida State University, Tallahassee, FL 32306-4510}}
\date{}
\maketitle
\begin{abstract}
We study monoidal 2-categories and bicategories in terms of categorical extensions and the cohomological data they determine in appropriate cohomology theories with coefficients in Picard groupoids. In particular, we analyze the hierarchy of possible commutativity conditions in terms of progressive stabilization of these data. We also show that monoidal structures on bicategories give rise to biextensions of a pair of (abelian) groups by a Picard groupoid, and that the progressive vanishing of obstructions determined by the tower of commutative structures corresponds to appropriate symmetry conditions on these biextensions. In the fully symmetric case, which leads us fully into the stable range, we show how our computations can be expressed in terms of the cubical $Q$-construction underlying MacLane (co)homology.
\end{abstract}

\setcounter{tocdepth}{2}
\tableofcontents

\section{Introduction}
\label{sec:introduction}

It is well known (see, e.g.\ \cite{Breen}, \cite{Baues-Conduche}) that monoidal categories can be analyzed using categorical extensions. If $\cE$ is a monoidal category, in fact a monoidal group-like groupoid, we denote by $B=\pi_{0}(\cE)$ and $A=\pi_{1}(\cE)$ (or simply $B=\pi_{0}$ and $A=\pi_{1}$ if $\cE$ is clear from the context) its group of isomorphism classes of objects and the automorphism group of the unit object, respectively. Then $\cE$ can be written as an extension of $B$ group by the Picard groupoid associated to the suspension $\Sigma A = [A\to 0]$. The resulting invariant is a class in $H^{2}(B,\Sigma A) = H^{3}(B,A)=H^{3}(K(B,1),A)$. If $\cE$ is braided (resp.\ symmetric), so that $B$ is necessarily abelian and has a trivial $A$-module structure, it is also well known that the corresponding invariant lives the higher Eilenberg-MacLane cohomology groups $H^{4}(K(B,2),A)$ (resp.\ $H^{5}(K(B,3),A)$). In the latter—symmetric—case, we refer to the invariant as the $k$-invariant of the extension.

Breen put forward an alternative approach to studying monoidal categories by considering their commutator~\cite{Breen}. If $\cE$ is as above, with $A$ and $B$ abelian and $B$ acting trivially on $A$, its commutator is an $A$-torsor $E$ over $B\times B$ whose fiber $E_{x,y}$ over $(x,y)\in B\times B$ is the set of all arrows $YX\to XY$, where $X,Y$ are objects of $\cE$ representing the isomorphism classes $x$ and $y$, respectively. Remarkably, $E$ carries two compatible partial composition laws that give it the structure of a (weak) biextension. This is, in a somewhat imprecise but suggestive manner, an extension of $B$  by $A$ with respect to each of the two variables while keeping the other fixed \cite{SGA7,Mumford}. To give an idea, if $f\colon YX\to XY$, $f'\colon YX'\to X'Y$, and $g\colon Y'X\to XY'$, we define $f+_{1}f'$ and $f +_{2} g$ as the composites:
\begin{gather*}
  Y(XX')\lto (YX)X' \xrightarrow{Xf} (XY)X' \lto X(YX') \xrightarrow{Xf'} X(X'Y) \lto (XX')Y\\
  \intertext{and}
 (YY')X \lto Y(Y'X) \xrightarrow{Yg} Y(XY') \lto (YX)Y' \xrightarrow{fY'} (XY)Y'\lto X(YY'),
\end{gather*}
where the unnamed arrows are the associator isomorphisms. The ``weak'' attribute refers to the fact that, for a general monoidal category, the composition laws may not be commutative, even though $A$ and $B$ are, analogously to what happens in ordinary central extensions. The two composition laws are compatible by way of an ``interchange law'' that we will full explain below, as it plays a crucial role in our work, but for now let us note that, if $g'\colon Y'X'\to X'Y'$, we have the identity
\begin{equation*}
  (f +_{1} f') +_{2} (g +_{1} g') = (f +_{2} g) +_{1} (f' +_{2} g')
\end{equation*}
expressing the compatibility between the two partial composition laws.

Analyzing the biextension associated to a monoidal category $\cE$ allows to directly characterize the conditions under which it admits a symmetric structure. To briefly explain this point, let us recall that a biextension $E\to B\times B$ is \emph{anti-symmetric} if the symmetric biextension $E \wedge \sigma^{*}E$ is trivial as a symmetric biextension, where $\sigma\colon B\times B\to B\times B$ is the permutation map. If $\Delta\colon B \to B\times B$ is the diagonal, consider the $A$-torsor $\Delta^{*}E$, which is an extension of $B$ by $A$. Its square $(\Delta^{*}E)^{2}$ is canonically split if $E$ is anti-symmetric, and we further say that $E$ is \emph{alternating} if $\Delta^{*}E$ itself is split, in a manner compatible with the canonical splitting of $(\Delta^{*}E)^{2}$.

One of the fundamental results in \cite{Breen} is that the monoidal structure on $\cE$ is symmetric precisely when the associated biextension given by the commutator map is alternating \emph{and} the underlying biextension admits a trivialization compatible with the anti-symmetric structure. Biextensions are amenable to descriptions in terms of cocycles, and in the particular case of that associated to a symmetric monoidal category there is a tight relation with the symmetry data. Note that if $\cE$ is symmetric monoidal,  the associated biextension, by its very definition, is trivialized by the symmetry isomorphism $c_{X,Y}\colon YX\to XY\in E_{x,y}$, and the extent to which it is a trivialization of the full alternating structure gives a direct access to the $k$-invariant. While the main motivation in \loccit is to provide a geometric characterization for certain components of $H^{3}(B,A)$ (and more generally $H^{n}(B,A)$), analyzing the biextension gives an alternative approach to the decomposition of the extension $0 \to \Sigma A \to \cE \to B \to 0$ that proves useful when generalized to a higher categorical setting.

Our goal in this paper is to apply these ideas to analyze monoidal bicategories. To do so, we need to appropriately upgrade the notions of extensions and biextensions to their categorical analogs by considering extensions and biextensions of groups by a Picard category.

It is known that the Q complex gives stable cohomology of Eilenberg-MacLane space, so we work out the cocycle conditions explicitly in case of full symmetry and verify the results that we found using biextensions. The main theorems of this paper are Theorem~\ref{thmExt} and Theorem~\ref{FinalThm} where we describe the cocycles of $H^{3}(\pi_0{\mathbb{E}},\cAut_{\mathbb{E}}(I))$ and invariants of a biextension, for a monoidal bi-category $\mathbb{E}$, explicitly.

\subsection*{Structure of the paper}
We have collected various preliminary items related to Picard groupoids in section \ref{background}, including torsors, contracted product of the torsors, and cohomology with values in a Picard groupoid. Aside from the background, we can divide the rest in two ways. One way is dividing the analysis of monoidal categories and monoidal bi-categories, and the other is the analysis using extensions and biextensions. A reader can check the analysis of only monoidal categories by reading sections \ref{ExtMonCat}, \ref{biExtMonCat}, \ref{biExtSymMonCat}, similarly, a reader can check the analysis of only monoidal bi-categories by reading sections \ref{ExtMonbiCat}, \ref{biExtMonbiCat}, \ref{biExtSymMonbiCat}, and \ref{HML}. The paper is structured in the natural order of increment in the symmetric structure of the categories. We first analyze the associativity using extensions in section \ref{Extensions}, then commutativity using biextensions in section \ref{biExt}, and finally we verify the results using MacLane cohomology in the fully symmetric case in section \ref{fullSym}. 

To manage the space better, we have added some diagrams at the end in appendix \ref{diagrams}. 

\subsection*{Acknowledgments}

We wish to thank Niles Johnson for discussions and helpful comments.

\section{Background}
\label{background}

In this section, we recall some structures related to Picard groupoids, such as torsors for a Picard groupoid, and cohomology with values in a Picard groupoid.
    
\subsection{Torsors for a Picard groupoid}
\label{sec:tors-picard-group}

If $G$ is a group, or more generally a group-object in a topos, the category of $G$-torsors, or principal homogeneous spaces, is a well-studied entity in Algebraic Geometry and Topology. Here we recall the notion of $\mathscr{A}$-torsor, where $\mathscr{A}$ is a Picard groupoid, and we discuss some properties of the 2-category $\Tors(\cA)$, following ref.\ \cite{number_theory}, in which the topic is discussed in great detail. 

\begin{definition}
  A Picard groupoid is a symmetric, group-like groupoid.
\end{definition}
\begin{lemma}
  Picard groupoids form a 2-Category $\Pic$.
\end{lemma}
\begin{proof}
  Morphisms are monoidal functors compatible with the symmetric structure. That is, if $\cA$ and $\cB$ are Picard groupoid, a morphism is a pair $(F,\lambda)$ where $F\colon \cA \to \cB$ is a functor and $\lambda_{x,y}$ is a family of natural isomorphisms
  \begin{equation*}
    \lambda_{x,y} \colon F(x) + F(y) \lto F(x + y) 
  \end{equation*}
  which is compatible with the associativity and commutativity data in the obvious way. 2-morphisms are monoidal natural transformations, i.e., the natural transformation $\theta\colon  F_{1} \to F_{2}$ such that the following diagram commutes:
\begin{equation*}
  \begin{tikzcd}
    F_{1}(x)+F_{1}(y) \arrow[d, "\theta+\theta"'] \arrow[r,"\lambda^{1}_{x,y}"]
    & F_{1}(x+y) \arrow[d, "\theta"] \\
    F_{2}(x)+F_{2}(y) \arrow[r,"\lambda^{2}_{x,y}"] & F_{2}(x+y)
  \end{tikzcd}
\end{equation*}
Details can be found in \cite{SM2Cat}.
\end{proof}

\begin{definition}
  Let $\cA$ be a Picard groupoid. An $\cA$-torsor
  $\cL$ is a module category over $\cA$, i.e., there
  is a bifunctor:
  \begin{equation*}
    +\colon \cA\times \cL\to \cL 
  \end{equation*}
  together with natural isomorphism:
  \begin{equation*}
    a_{x,y,v}\colon (x+y)+v \cong x+(y+v), x,y\in \cA, v\in \cL, 
  \end{equation*}
  satisfying
  \begin{enumerate}
  \item the pentagon axiom,
  \item for any $x\in \cA$, the functor from $\cL$ to
    $\cL$ given by $v\mapsto x+v$ is an equivalence,
  \item for any $v\in \cL$, the functor from $\cA$ to
    $\cL$ given by $x\mapsto x+v$ is an equivalence of
    categories.
  \end{enumerate}
\end{definition}

\begin{definition}
If $\cL_{1}, \cL_2$ are $\cA$-torsors, then $\text{Hom}_{\cA}(\cL_{1},\cL_{2})$ is the category defined as follows. Objects are 1-morphisms, i.e., equivalences $F\colon \cL_{1}\to \cL_{2}$ together with isomorphism $\lambda\colon F(x+v)\cong x+F(v)$ such that the following diagram commutes:
\begin{equation*}
\begin{tikzcd}
F((x+y)+v) \arrow[d] \arrow[r] & (x+y)+F(v) \arrow[d] \\
F(x+(y+v)) \arrow[r]           & x+(y+F(v))          
\end{tikzcd}
\end{equation*}
Morphisms are natural transformations $\theta\colon F_1\to F_2$ such that the following diagram commutes:
\begin{equation*}
\begin{tikzcd}
F_{1}(x+v) \arrow[d, "\theta"'] \arrow[r] & x+F_{1}(v) \arrow[d, "\theta"] \\
F_{2}(x+v) \arrow[r]                      & F_{2}(x+v)                    
\end{tikzcd}
\end{equation*}
\end{definition}
\begin{remark}
\begin{enumerate}
\item There is an equivalent definition of $\cA$-torsor. $\cL$ is an $\cA$-torsor, if $\cL$ is a module category over $\cA$ such that the map $(-, p_2)\colon \cA\times \cL\to \cL\times \cL$ is a natural equivalence, where $p_2$ is the canonical projection map.
\item We have defined the notion of left torsors. That of right torsors is defined in the same way. If $\cL$ is a left $\cA$-torsor with an action $+_L$, one can define $\cL$ as a right torsor with an action $v +_R x \colon = x^{-1} +_L v$ for $v \in \cL, x \in \cA$. 
\end{enumerate}
\end{remark}
\begin{remark}
\begin{enumerate}
\item If $\cL_{1}, \cL_{2}$ are $\cA$-torsors, then $\text{Hom}_{\cA}(\cL_{1},\cL_{2})$ forms a category, i.e., $\cA$-torsors form a 2-category.
\item Moreover, it is a category enriched over itself, i.e., if $\cL_{1}, \cL_{2}$ are $\cA$-torsors, then $\text{Hom}_{\cA}(\cL_{1},\cL_{2})$ is also a $\cA$-torsor. (See \cite{number_theory} for the details).
\end{enumerate}
\end{remark}
One does not require Picard condition or even braiding to define torsors, i.e., one can similarly define torsors for group-like groupoids. (See, e.g. \cite{BreenFrench}).
\subsubsection*{Contracted product}
We will need to consider the notion of the contracted product of torsors for Picard groupoids in some detail. It is introduced in \cite{BreenFrench}.
\begin{definition}
Let $\cA$ be a Picard groupoid, $\cB$ and $\mathscr{C}$ be $\cA$-torsors, then the contracted product is the category, $\cB\wedge^{\cA}\mathscr{C}$ as follows:

Objects are $(b,c)$ for $b \in \text{Ob}(\cB), c\in \text{Ob}(\mathscr{C})$, morphisms are equivalence classes of triples $[(f,h,g)]$, where $h \in \text{Ob}(\cA)$, $f\colon  b \to b'\cdot h$, and $g\colon h\cdot c\to c'$ are morphisms of $\cB$, and $\mathscr{C}$, respectively. Two triples $(f,h,g)$ and $(f',g',h')$ are equivalent if there is a morphism $\gamma\colon  h\to h'$ such that the diagrams commute:
\begin{equation*}
\begin{tikzcd}
&  & b'\cdot h \arrow[dd, "id_{b'}\cdot \gamma"] &  & h\cdot c \arrow[dd, "\gamma\cdot id_{c}"'] \arrow[rrd, "g"] &  &    \\
b \arrow[rrd, "f'"'] \arrow[rru, "f"] &  &                                             &  &                                                             &  & c' \\
&  & b'\cdot h'                                  &  & h'\cdot c \arrow[rru, "g'"']                                &  &   
\end{tikzcd}
\end{equation*}
The composition of two morphisms $(f,h,g)\colon  (b,c)\to (b',c')$ and $(f',h',g')\colon (b',c')\to (b'',c'')$ is defined as follows: 
\begin{equation*}
b \xrightarrow{f} b'\cdot h \xrightarrow{f'\cdot h} (b''\cdot h')\cdot h \xrightarrow{\sim} b''\cdot(h'\cdot h)
\end{equation*}
\begin{equation*}
c \xrightarrow{g} c'\cdot h \xrightarrow{g'\cdot h} (c''\cdot h')\cdot h \xrightarrow{\sim} c''\cdot(h'\cdot h)
\end{equation*}
and $h'\cdot h$.

$\cB\wedge^{\cA} \mathscr{C}$ is also an $\cA$-torsor with the action:
\begin{equation*}
h\cdot(b,c) = (b\cdot h,c).
\end{equation*}
\end{definition}
\begin{remark}
For any $b,c,h$, we have the isomorphism $(b\cdot h,c) \xrightarrow[]{\sim} (b,h\cdot c)$ by the triple $(id_{b\cdot h},h,id_{h\cdot c})$
\end{remark}

\subsection{Cohomology with values in a Picard category}
\begin{definition}
Let $\cA_{i}$'s be Picard groupoids. A complex of Picard groupoids $\cA_{\bullet}$ looks like:

\begin{equation*}
\begin{tikzcd}
& {}                                                                                        &                                                                                                            & {}                                                                                 &                                                                                              & {}                                                                &                                       &        \\
\cA_0 \arrow[r, "\delta"] \arrow[r] \arrow[rr,"0", bend left=49] & \cA_1 \arrow[r, "\delta"] \arrow[rr,"0"', bend right=49] \arrow[u, "\chi_0", Rightarrow, shorten <= 0em, shorten >= 0.5em] & \cA_2 \arrow[r, "\delta"] \arrow[r] \arrow[rr, dashed,"0", bend left=49] \arrow[d, "\chi_1", Rightarrow, shorten <= 0em, shorten >= 0.5em] & \cdots \arrow[r, "\delta"] \arrow[rr,"0"', bend right=49] \arrow[u, "\chi_{2}", Rightarrow, shorten <= 0em, shorten >= 0.5em] & \cA_{n-1} \arrow[r, "\delta"] \arrow[rr,"0", bend left=49] \arrow[d, "\chi_{n-2}", Rightarrow, shorten <= 0em, shorten >= 0.5em] & \cA_{n} \arrow[r, "\delta"] \arrow[u, "\chi_{n-1}", Rightarrow, shorten <= 0em, shorten >= 0.5em] & \cA_{n+1} \arrow[r, "\delta"] & \cdots \\
&                                                                                           & {}                                                                                                         &                                                                                    & {}                                                                                           &                                                                   &                                       &       
\end{tikzcd}
\end{equation*}
\vspace{-0.15in}
A chain of monoidal functors that satisfies the \textit{wave}:

\begin{equation*}
\begin{tikzcd}
& {}                                                                                                                                        & {}                                                                                                                      &    \\
\cA_n \arrow[r, "\delta"] \arrow[r] \arrow[rr, "0", bend left=49] \arrow[rrr, "0", bend left=60] \arrow[rrr, "0"', bend right=60] & \cA_{n+1} \arrow[r, "\delta"] \arrow[rr, "0"', bend right=49] \arrow[u, "\chi_n", shorten <= 0em, shorten >= 0.5em, Rightarrow] \arrow[d, "\text{can}", Rightarrow] & \cA_{n+2} \arrow[r, "\delta"] \arrow[r] \arrow[d, "\chi_{n+1}", shorten <= 0em, shorten >= 0.5em, Rightarrow] \arrow[u, "\text{can}", Rightarrow] & {} \\
& {}                                                                                                                                        & {}                                                                                                                      &   
\end{tikzcd}
\end{equation*}
\end{definition}

\begin{definition}[Cohomology of a complex of Picard groupoids]\cite{Ulbrich}
Given a complex $\cA_{\bullet}$ of Picard groupoid as above, we define the cohomology in steps as follows:
\begin{enumerate}
\item A category of $n$-pseudococycles $\mathscr{P}^{n}(\mathscr{A_{\bullet}})$ is defined as $\{(P,g) | P \in \cA_{n}, g\colon \delta(P)\to I \in \text{Hom}_{\cA_{n+1}}\}$    
\item A subcategory of $n$-cocycles $\cL^{n}(\cA_{\bullet})$ of $\mathscr{P}^{n}(\cA_{\bullet})$ has the objects for which 
    
\begin{equation*}
\begin{tikzcd}
I_{n+2} \arrow[r, "\chi^{-1}_{n}"] & \delta(\delta(P)) \arrow[r, "\delta(g)"] & \delta(I_{n+1}) \arrow[r] & I_{n+2}\  = \ I_{n+2} \arrow[r, "id"] & I_{n+2}
\end{tikzcd}
\end{equation*}
\item A subcategory of $n$-coboundaries $\cB^{n}(\cA_{\bullet})$ of $\cL^{n}(\cA_{\bullet})$ of objects of the form \begin{equation*}
(\delta(Q),\chi_{Q}), Q\in \cA_{n-1}.
\end{equation*}
\item Finally we define the cohomology of the complex as the Picard groupoid 
\begin{equation*}
\mathscr{H}^{n}(\cA_{\bullet})\colon = \cL^{n}(\cA_{\bullet})/\cB^{n}(\cA_{\bullet}).
\end{equation*}
\end{enumerate}
\end{definition}
\begin{remark}
For our purposes, we will consider the cohomology group $H^{n}(\mathscr{C_{\bullet}}) \colon = \pi_{0}(\mathscr{H}^{n}(\cA_{\bullet})).$
\end{remark}
Using this definition of cohomology of a complex of Picard categories we can define a group cohomology with values in a Picard category as follows.

Let $\cA$ be a Picard category. Given any simplicial set $X_{\bullet}$, consider the simplicial category $\cA^{X_{\bullet}}$, where $\cA^{X_{i}} = \text{Funct}(X_{i}\to \cA)$ by letting $X_{i}$ as a discrete category.

\begin{equation*}
\begin{tikzcd}
\cA^{X_0} \arrow[r, shift left] \arrow[r, shift right] & \cA^{X_1} \arrow[r, shift left=2] \arrow[r, shift right=2] \arrow[r] & \cA^{X_2} \arrow[r, shift left=3] \arrow[r, shift left] \arrow[r, shift right] \arrow[r, shift right=3] & \cdots
\end{tikzcd}
\end{equation*}

Note that we can borrow the Picard groupoid structure from $\cA$ to make $\cA^{X_{i}}$ a Picard groupoid too. As shown in \cite{CohomInPic}, by taking alternating sums we obtain a complex of Picard groupoids:

\begin{equation*}
\begin{tikzcd}
& {}                                                                                                  &                                                                           &        \\
\cA^{X_0} \arrow[r, "\partial_0"] \arrow[rr, bend left=49] & \cA^{X_1} \arrow[r, "\partial_1"] \arrow[rr, bend right=49] \arrow[u, "\chi_0", Rightarrow, shorten <= 0em, shorten >= 0.5em] & \cA^{X_2} \arrow[r, "\partial_2"] \arrow[d, "\chi_1", Rightarrow, shorten <= 0em, shorten >= 0.5em] & \cdots \\
&                                                                                                     & {}                                                                        &       
\end{tikzcd}
\end{equation*}

with the boundary maps as follows:
\begin{equation*}
\partial_{n} = \sum_{i = 0}^{n+1}(-1)^{i}d_{n}^{i^{*}}.
\end{equation*}
  
Where $d^{i^{*}}_{n}\colon  \cA^{X_{n}} \to \cA^{X_{n+1}} \colon  (X_{n}\xrightarrow{f} \cA) \mapsto (X_{n+1}\xrightarrow{d^{i}_{n+1}} X_{n}\xrightarrow{f} \cA)$

We denote the cohomology of this complex as $H^{n}(X_{\bullet},\cA)$.

\begin{definition}\label{defCohomoInPic}
Let $G$ be a group, $\cA$ be a Picard groupoid we define the cohomology $H^{n}(G,\cA) = H^{n}(\textbf{B}G,\cA)$, where $\textbf{B}G$ is the bar construction.
\end{definition}

\section{Categorical Extensions}\label{Extensions}
It is well-known that one can analyze a monoidal category by considering a categorical extension of its $\pi_{0}$ group by the Picard category $\Sigma(\pi_{1})$ and analyzing the K-invariant, i.e., some cohomology groups of the extension. We recall the relevant details quickly and develop the theory for monoidal bi-categories.
\begin{definition}
\begin{enumerate}
    \item Let $G$ be a group, $\cA$ be a Picard groupoid. A monoidal category $\mathscr{E}$ is called a categorical extension of $G$ by $\cA$ if there exists an exact sequence of functors, i.e., concatenation of two functors is homotopically zero, as follows:
\begin{equation*}
0 \to \cA \to \mathscr{E} \to G \to 0 
\end{equation*}
where $G$ is viewed as a discrete category.
\item $G$ acts canonically on $\cA$. If this action is trivial, we call the extension a central extension.
\end{enumerate}
\end{definition}
In our case, the action of $G$ will always be trivial on $\cA$.  

\begin{example}\label{CatExtEg}
Given a functor from a Picard groupoid $\cA$ to a monoidal category $\mathscr{E}$, we can consider a group $G = \pi_{0}(\mathscr{E})/\pi_{0}(\cA)$, where $\pi_{0}(-)$ is the group generated by the isomorphism classes of the objects of the category. Then this creates a categorical extension of $G$ by $\cA$.
\end{example}
Group extensions have a nice property that the category of extensions of a group $G$ by an abelian group $A$, is monoidal, and its $\pi_{0}$ group is isomorphic to $H^{2}(G,A)$. In the case of categorical extensions, we can get the classification in a similar way. Let $G$ be group, and $\cA$ be a Picard category, then Ext($G,\cA$) is a monoidal category, and  $\pi_{0}(\text{Ext}(G,\cA))\cong H^{2}(G,\cA)$. (See, e.g. \cite{Ulbrich},\cite{CohomInPic}).

\subsection{Categorical Extensions and Monoidal categories}\label{ExtMonCat}
Starting with a monoidal category $\mathscr{E}$, we can consider the Picard groupoid $\cA = \Sigma (\pi_{1}(\mathscr{E}))$, where $\pi_{1}(\mathscr{E})$ is the automorphism group of the monoidal unit of $\mathscr{E}$,  $ \text{Aut}_{\mathscr{E}}(I)$. So we can consider a categorical extension as in example \ref{CatExtEg}. Notice that $G$ now becomes $\pi_{0}(\mathscr{E})$, as $\pi_{0}(\cA)$ is now trivial. 
\begin{equation*}
0\to \Sigma (\pi_{1}(\mathscr{E}))\to \mathscr{E}\to \pi_{0}(\mathscr{E})\to 0.
\end{equation*}
So this corresponds with a cocycle in $H^{2}(\pi_{0}(\mathscr{E}),\Sigma(\pi_{1}(\mathscr{E})))$, which is isomorphic to $H^{3}(\pi_{0}(\mathscr{E}),\pi_{1}(\mathscr{E}))$. One can explicitly construct the cocycle as shown in \cite{BreenSchreier}. We recall the cocycles for the ease of readers.

\subsubsection*{Cohomology}\label{Cohomology of Extensions}
Consider an arbitrary section $s\colon \pi_{0}(\mathscr{E}) \to \mathscr{E}$ of $p\colon \mathscr{E} \to \pi_{0}(\mathscr{E})$. Due to the monoidal structure of $\mathscr{E}$, for each $x,y \in \pi_{0}(\mathscr{E}),$ there exists an isomorphism $ c_{x,y}\colon s(x)\otimes s(y)\xrightarrow{\cong} s(xy)$. As the monoidal product is associative up to an isomorphism, there exists an $f_{x,y,z} \in \pi_{1}(\mathscr{E})$. Since the association has to satisfy the pentagon axiom, $f$ must also satisfy a pentagon which is the cocycle condition. Hence $[f] \in H^{3}(\pi_{0}(\mathscr{E}), \pi_{1}(\mathscr{E}))$. This is the cocycle that corresponds to the monoidal category $\mathscr{E}$ we started with. 
         
\subsection{Categorical Extensions and Monoidal bi-categories}\label{ExtMonbiCat}
To analyze monoidal bi-categories, we upgrade the abelian groups to Picard groupoids. Now we see the definition of the fundamental groupoid and prove that it is indeed a Picard groupoid in this case.
\begin{definition}
Let $\mathbb{E}$ be a monoidal bi-category, $\cA = \cAut_{\mathbb{E}}(I)$ be the category of automorphisms of the monoidal unit. Then for each object $X \in \mathbb{E}$, the right unitor, $r_{X}$ induces the functor $\eta_{X}\colon \cA\to \cAut_{\mathbb{E}}(X)$ as follows:
\begin{equation*}
\begin{tikzcd}
X \arrow[r, "\eta_{X}(f)"] \arrow[rd, "\zeta_{f}", Rightarrow] & X                      \\
XI \arrow[u, "r_{X}"] \arrow[r, "Xf"']                         & XI \arrow[u, "r_{X}"']
\end{tikzcd}
\end{equation*}
along with the following 2-morphism, for a $g\colon X \to Y$ in $\mathscr{C}$:    
\begin{equation*}
\begin{tikzcd}
& \cA \arrow[ld, "\eta_{X}"'] \arrow[rd, "\eta_{Y}"] & {}               \\
\cAut(X) \arrow[rr, "g_{*}"'] \arrow[rru, "\mu_{g}"', Rightarrow, shorten <= 1.7em, shorten >= 2.8em,xshift = -15,yshift = -5] &                                                            & \cAut(Y)
\end{tikzcd}
\end{equation*}

(Here $\mu_{g}\colon  g\eta_{X}(-)g^{\bullet} \Rightarrow \eta_{Y}(-)$).    
\end{definition}

\begin{lemma}
    Let $\mathbb{E}$ be a monoidal bi-category, using properties of $r_{X}$, the following coherence conditions for $\eta_{X}$ hold. 
\begin{enumerate}
    \item For $f,f'\colon I\to I$, $\alpha\colon  f \Rightarrow f'$:
\begin{equation*}
\begin{tikzcd}
&                                      &                         &   &                                                                                                                        & {} \arrow[dd, "\eta_{X}(\alpha)", Rightarrow, shorten <= 1.5em, shorten >= 1.5em] &               \\
X \arrow[rr, "\eta_{X}(f)"] \arrow[rrdd, "\zeta_{f}", Rightarrow, shorten <= 0em, shorten >= 0.5em, yshift = 5]                 &                                      & X                       &   & X \arrow[rr, "\eta_{X}(f)", bend left] \arrow[rr, "\eta_{X}(f')"', bend right] \arrow[rrdd, "\zeta_{f'}"', Rightarrow, Rightarrow, shorten <= 0.3em, shorten >= 0em, yshift = -5] &                                               & X             \\
& {} \arrow[dd, "X\alpha", Rightarrow, shorten <= 1.2em, shorten >= 1.2em] &                         & = &                                                                                                                        & {}                                            &               \\
XI \arrow[uu, "r_{X}"] \arrow[rr, "Xf", bend left] \arrow[rr, "Xf'"', bend right] &                                      & XI \arrow[uu, "r_{X}"'] &   & XI \arrow[uu] \arrow[rr, "Xf'"']                                                                                       &                                               & XI \arrow[uu] \\
& {}                                   &                         &   &                                                                                                                        &                                               &              
\end{tikzcd}
\end{equation*}
\item For $g,g'\colon  X \to Y$ and $\beta\colon  g \Rightarrow g'$:
\begin{equation*}
\begin{tikzcd}
& \cA \arrow[ld, "\eta_{X}"'] \arrow[rd, "\eta_{Y}"] & {}               &   &                                                                          & \cA \arrow[ld, "\eta_{X}"'] \arrow[rd, "\eta_{Y}"] & {}               \\
\cAut(X) \arrow[rr, "g'_{*}"] \arrow[rru, "\mu_{g'}", Rightarrow, shorten <= 1.2em, shorten >= 3.7em] \arrow[rr, "g_{*}"', bend right=49] &                                                            & \cAut(Y) & = & \cAut(X) \arrow[rr, "g_{*}"'] \arrow[rru, "\mu_{g}"', Rightarrow, shorten <= 1.2em, shorten >= 2.4em,xshift = -15, yshift = -8] &                                                            & \cAut(Y) \\
& {} \arrow[uu, "\beta_{*}", Rightarrow, shorten <= 2.2em, shorten >= 1.2em,yshift = -20]                     &                  &   &                                                                          &                                                            &                 
\end{tikzcd}
\end{equation*}
\item For $X\xrightarrow{g} Y\xrightarrow{h} Z$ then the following diagram commutes:
\begin{equation*}
\begin{tikzcd}
& \cA \arrow[rdd, "\eta_{Z}"] \arrow[ldd, "\eta_{X}"']                &                  &   &                                                                        & \cA \arrow[ldd, "\eta_{X}"'] \arrow[rdd, "\eta_{Z}"] \arrow[ddd, "\eta_{Y}"] &                  \\
& {}                                                                          &                  & = &                                                                        &                                                                                      & {}               \\
\cAut(X) \arrow[rr, "(gf)_{*}"] \arrow[rd, "f_{*}"'] \arrow[ru, "\mu_{gf}"', Rightarrow,xshift = 25, yshift = 5] & {}                                                                          & \cAut(Z) &   & \cAut(X) \arrow[rd, "f_{*}"'] \arrow[r, "\mu_{f}", Rightarrow] & {}                                                                                   & \cAut(Z) \\
& \cAut(Y) \arrow[ru, "g_{*}"'] \arrow[u, "{\rho_{f,g}}", Rightarrow] &                  &   &                                                                        & \cAut(Y) \arrow[ru, "g_{*}"'] \arrow[ruu, "\mu_{g}"', Rightarrow, shorten <= 1em, shorten >= 2.5em,yshift = 10, xshift = -13]             &                 
\end{tikzcd}
\end{equation*}
Here $\rho_{f,g}$ exists as $f^{\bullet}g^{\bullet} \cong (gf)^{\bullet}$, but need not be equal on the nose.

\begin{enumerate}
    \item In the case of a chain of morphisms:
\begin{equation*}
X\xrightarrow{f} Y\xrightarrow{g}Z\xrightarrow{h} W 
\end{equation*}
We get the following commutative diagram of $\rho_{-,-}\colon$ 
\begin{equation*}
\begin{tikzcd}
& \cAut(W)                                                            &                                                                                &   &                                                                & \cAut(W)                                                                                        &                                                                               \\
& {}                                                                          &                                                                                & = & {}                                                             &                                                                                                         &                                                                               \\
\cAut(X) \arrow[rr, "(gf)_{*}"] \arrow[rd, "f_{*}"'] \arrow[ruu, "(hgf)_{*}"] & {}                                                                          & \cAut(Z) \arrow[luu, "h_{*}"'] \arrow[lu, "{\rho_{gf,h}}", Rightarrow,xshift = -20, yshift = 5] &   & \cAut(X) \arrow[ruu, "(hgf)_{*}"] \arrow[rd, "f_{*}"'] & {}                                                                                                      & \cAut(Z) \arrow[luu, "h_{*}"'] \arrow[l, "{\rho_{g,h}}"', Rightarrow] \\
& \cAut(Y) \arrow[ru, "g_{*}"'] \arrow[u, "{\rho_{f,g}}", Rightarrow] &                                                                                &   &                                                                & \cAut(Y) \arrow[ru, "g_{*}"'] \arrow[uuu, "(hg)_{*}"] \arrow[luu, "{\rho_{f,hg}}", Rightarrow, shorten <= 1em, shorten >= 2.5em,yshift = 10, xshift = 13] &                                                                              
\end{tikzcd}
\end{equation*}
\begin{equation*}
\rho_{f,g}\cdot \rho_{gf,h} = \rho_{g,h}\cdot \rho_{f,hg}.
\end{equation*}
\end{enumerate}
\begin{equation*}
\begin{tikzcd}
(XY)I \arrow[r, "(XY)f"] \arrow[d, "\cong"'] \arrow[dd, "r_{XY}"', bend right=65] & (XY)I \arrow[d, "\cong"] \arrow[dd, "r_{XY}", bend left=65]    &  &   & (XY)I \arrow[r, "(XY)f"] \arrow[dd, "r_{XY}"'] & (XY)I \arrow[dd, "r_{XY}"] \arrow[ldd, "\zeta_{f}", Rightarrow] \\
X(YI) \arrow[d, "Xr_{Y}"'] \arrow[r, "X(Yf)"]                                     & X(YI) \arrow[d, "Xr_{Y}"] \arrow[ld, "X\zeta_{f}", Rightarrow] &  & = &                                                &                                                                 \\
XY \arrow[r, "X\eta_{Y}(f)"']                                                     & XY                                                             &  &   & XY \arrow[r, "X\eta_{Y}(f)"']                  & XY                                                             
\end{tikzcd}
\end{equation*}
This implies: $\eta_{XY} = X\eta_{Y} $
\end{enumerate}
\end{lemma}
    
    Using the results from this lemma, we are now set to prove that the fundamental groupoid, $\Pi_{1}(\mathbb{E}) = \cAut_{I}(\mathscr{E})$ is a Picard groupoid.
    \begin{lemma}
        Let $\mathbb{E}$ be a monoidal bi-category, then $\Pi_{1}(\mathbb{E}) = \cAut_{\mathbb{E}}(I)$ is a Picard groupoid.
    \end{lemma}
    \begin{proof}
    Let $f,g \in \text{Ob}(\cA)$, then we define a monoidal product $f\cdot g$ by the composition $g\circ f$. Being group-like groupoid is trivial to verify. To check that it is symmetric, notice that $\eta_{I}$ is identity. So for any $f,g \in \text{Ob}(\cA)$, we get $\mu_{f} (g)\colon  f\eta_{I}(g)f^{\bullet} \to \eta_{I}(g)$. This in turn gives $\mu_{f}(g)f\colon f\eta_{I}(g)f^{\bullet}f \cong f\eta_{I}(g)\to \eta_{I}(g)f$. So $\eta_{I}$ being an identity gives us a symmetry isomorphism between $fg$ and $gf$. 
    \end{proof}
    \begin{remark}
    Since $\cA$ is a Picard groupoid, the entire theory can be developed by considering left unitors instead of right unitors while defining the functors $\eta$.
    \end{remark}
    \begin{lemma}
    Let $\mathbb{E}$ be a monoidal bi-Category, then for any $X \in \text{Ob}(\mathbb{E})$, $\cAut_{\mathbb{E}}(X)$ is an $\cAut_{\mathbb{E}}(I)$-torsor.
\end{lemma}
\begin{proof}
  We will show that $\cA = \cAut_{\mathbb{E}}(I)$ acts on $\cAut_{\mathbb{E}}(X)$ satisfying the required conditions. Consider, $+\colon \cA \times \cAut_{\mathbb{E}}(X) \to \cAut_{\mathbb{E}}(X)$ such that $(h,f) \mapsto f\circ \eta_{X}(h). $ We will call this a left action. One can similarly define a right action by post-composing with $\eta_{X}(h)$ instead.
\end{proof}
\begin{lemma}
    Since $\cA$ is a Picard groupoid, left and right actions by $\cA$ are equivalent. 
\end{lemma}
\begin{proof}
The following diagram commutes due to the coherence of the bi-category structure and gives us the required result.
\begin{equation*}
\begin{tikzcd}
X \arrow[r, "f"] \arrow[rrr, "\eta_{X}(h)"', bend right=49] & X \arrow[r, "\eta_{X}(h)"]           & X \arrow[r, "f^{\bullet}"] & X \arrow[r, "f"] & X \\
& {} \arrow[ru, "\mu_{f}", Rightarrow] &                            &                  &  
\end{tikzcd}
\end{equation*}
\end{proof}

\subsubsection{Cohomology}\label{Cohomology of Extensions2}
Let $\mathbb{E}$ be a monoidal bi-category, then there exists a categorical extension as follows:
    
\begin{equation*} 0 \to \Sigma\cA \to \mathbb{E} \xrightarrow{p} \pi_{0}(\mathbb{E}) \to 0.\end{equation*} 
We can analyze this extension using a technique similar to the one we used for monoidal categories, i.e., we analyze the sections of map $p$. Let $s\colon  \pi_{0}(\mathbb{E}) \to \mathbb{E}$ be a section of $p\colon \mathbb{E}\to \pi_{0}(\mathbb{E})$.
The due to the monoidal product of $\mathbb{E}$, for each $x,y \in \pi_{0}(\mathbb{E}), \exists\ c_{x,y}\colon s(x)\otimes s(y)\xrightarrow{\cong} s(xy)$ along with
\begin{equation*}
\begin{tikzcd}
(s(x)\otimes s(y))\otimes s(z) \arrow[rr, "{c_{x,y}\otimes id_{s(z)}}"] \arrow[d, "{\alpha_{s(x),s(y),s(z)}}"'] &  & s(xy)\otimes s(z) \arrow[r, "{c_{xy,z}}"]  & s(xyz) \arrow[d, "{f_{x,y,z}}", blue] 
 \\
s(x)\otimes (s(y)\otimes s(z)) \arrow[rr, "{id_{s(x)}\otimes c_{y,z}}"']                                        &  & s(x)\otimes s(yz) \arrow[r, "{c_{x,yz}}"'] & s(xyz)                                                                       
\end{tikzcd}
\end{equation*}
 and the pentagon:
\begin{equation*}
\begin{tikzcd}
s(xyzt) \arrow[rr, "{f_{xy,z,t}}",blue] \arrow[d, "{f_{x,y,z}\otimes id_{t}}"',blue] &                                                  & s(xyzt) \arrow[d, "{f_{x,y,zt}}",blue] \arrow[lldd, "{\theta_{x,y,z,t}}", shorten <= 0em, shorten >= 1.5cm, Rightarrow] \\
s(xyzt) \arrow[rd, "{f_{x,yz,t}}"',blue]                                        &                                                  & s(xyzt)                                                                          \\
{}                                                                         & s(xyzt) \arrow[ru, "{id_{x}\otimes f_{y,z,t}}"',blue] &                                                                                 
\end{tikzcd}
\end{equation*}
Such that this satisfies the associahedron $K_{5}$.
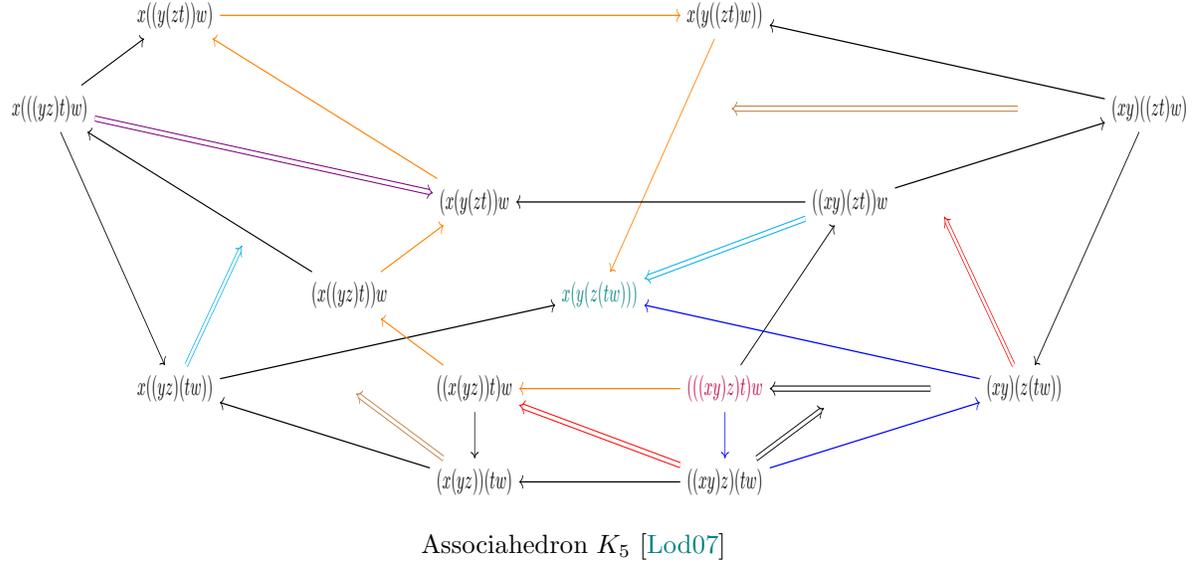
\begin{equation*}
\adjustbox{scale = {0.57}{1}}{%
\begin{tikzcd}
& x((y(zt))w) \arrow[rrrrr,orange]                         &    &                                      &                                                 &             & x(y((zt)w)) \arrow[lddd,orange]                                                           &                                                              &                           &                                                   &                                        \\
x(((yz)t)w) \arrow[ru] \arrow[rddd] \arrow[rrrrd, Rightarrow, violet] &                                                   &    &                                      &                                                 &             & {}                                                                                 &                                                              &                           & {} \arrow[lll, Rightarrow, brown]                        & (xy)((zt)w) \arrow[llllu] \arrow[lddd] \\
&                                                   & {} &                                      & (x(y(zt))w \arrow[llluu,orange]                        &             &                                                                                    & ((xy)(zt))w \arrow[rrru] \arrow[lll] \arrow[lld, Rightarrow, cyan] & {}                        &                                                   &                                        \\
&                                                   &    & (x((yz)t))w \arrow[ru,orange] \arrow[llluu] &                                                 & \textcolor{teal}{x(y(z(tw)))} &                                                                                    &                                                              &                           &                                                   &                                        \\
& x((yz)(tw)) \arrow[rrrru] \arrow[ruu, Rightarrow, shorten <= 0em, shorten >= 0.5cm, cyan] &    & {}                                   & ((x(yz))t)w \arrow[d] \arrow[lu,orange]                &             & \textcolor{purple}{(((xy)z)t)w} \arrow[d,blue] \arrow[ll,orange] \arrow[ruu]                                       & {}                                                           & {} \arrow[ll, Rightarrow] & (xy)(z(tw)) \arrow[llllu,blue] \arrow[luu, Rightarrow, red] &                                        \\
&                                                   &    &                                      & (x(yz))(tw) \arrow[lllu] \arrow[lu, Rightarrow, brown] &             & ((xy)z)(tw) \arrow[ll] \arrow[rrru,blue] \arrow[ru, Rightarrow, shorten <= 0em, shorten >= 0.5cm] \arrow[llu, Rightarrow, red] &                                                              &                           &                                                   &                                       
\end{tikzcd}}
\end{equation*}
\begin{center}
Associahedron $K_5$ \cite{Loday}
\end{center}

\begin{theorem}\label{thmExt}
Let $\mathbb{E}$ be a monoidal bi-category. With the definition of $f$ and $\theta$ as above, $(f,\theta) \in H^{3}\left(\pi_{0}(\mathbb{E}),\cA\text{ut}_{\mathbb{E}}(I)\right)$.
\end{theorem}
\begin{proof}
Using definition \ref{defCohomoInPic} and the bar constructions for $\pi_{0}(\mathbb{E})$, notice that  $f \in \cA^{G\times G\times G}$ and $\theta \in \text{Hom}_{\cA^{G\times G\times G\times G}}.$ This makes $(f,\theta) \in \mathscr{P}^{n}\left(\cA^{(\pi_{0}\mathbb{E})_{\bullet}}\right)$, i.e., $(f,\theta)$ is a 3-pseudococycle. Finally, it becomes a 3-cocycle, i.e., it belongs to $\cL ^{3}\left(\cA^{(\pi_{0}\mathbb{E})_{\bullet}}\right)$, by satisfying the same condition as the associahedron's commutativity. 
\begin{equation*}
\theta_{xy,z,t,w}\cdot \theta_{x,y,zt,w}\cdot \theta_{x,y,z,t} = \theta_{x,y,z,tw}\cdot \theta_{x,yz,t,w}\cdot \theta_{y,z,t,w}.
\end{equation*}
\end{proof}
\subsubsection{Long Exact Sequence of Cohomology Groups}
If $\cA$ is a Picard groupoid, there exists a short exact sequence of Picard groupoids:

\begin{equation*}
0\to \Sigma (\pi_{1}(\cA))\to \cA\to \pi_{0}(\cA)\to 0
\end{equation*}
For a group $G$, taking cochain complexes $C^{*}(G,-)$ gives a shirt exact sequence of chain complexes, and then we can consider the standard long exact sequence from it. The long exact sequence looks like the following, as defined in \cite{Ulbrich}.
\begin{equation*}
\begin{tikzcd}
\cdots \arrow[r] & { H^{n+1}(G,\pi_{1}(\cA))} \arrow[r] & { H^{n}(G,\cA)} \arrow[r] & { H^{n}(G,\pi_{0}(\cA))} \arrow[r] & { H^{n+2}(G,\pi_{1}(\cA))} \arrow[r] & \cdots     
\end{tikzcd}
\end{equation*}    
Here the extra shift of dimension comes from using the suspension isomorphism: $H^{n}(G_1,\Sigma G_2) \cong H^{n+1}(G_{1}, G_{2}).$

For a given monoidal bi-category $\mathbb{E}$, we get the following long exact sequence of cohomology groups:
\begin{equation*}
\begin{tikzcd}
\cdots \arrow[r] & { H^{n+1}(\pi_{0}\mathbb{E},\pi_{1}(\cA))} \arrow[r] & { H^{n}(\pi_{0}\mathbb{E},\cA)} \arrow[r] & { H^{n}(\pi_{0}\mathbb{E},\pi_{0}(\cA))} \arrow[r] & { H^{n+2}(\pi_{0}\mathbb{E},\pi_{1}(\cA))} \arrow[r] & \cdots \\
& {\text{For } n = 3:}                & {[f,\theta]} \arrow[r, maps to]  & {[f]}                                     &                                             &       
\end{tikzcd}
\end{equation*}
\section{Biextensions}\label{biExt}\label{biExtMonCat}
We use extensions to analyze the associativity of a monoidal (bi-)category, but to analyze the braiding and symmetry, we need a structure that takes two objects into account. Breen used this idea in \cite{Breen} to analyze the symmetry of monoidal categories using biextensions. With this, we also recover the conditions for associativity when we restrict ourselves to one group. Let us first recall the definition of biextensions in the case of groups, then we quickly recall Breen's work on monoidal categories that leads to our work on monoidal bi-categories.
\begin{definition}(See \cite{Mumford},\cite{SGA7})
Let $A$ be an abelian group, $G, H$ be groups. A biextension $E$, of $G\times H$ by $A$ is an $A$-torsor over $G\times H$, endowed with a pair of partial composition laws whose restriction to appropriate fibers may be depicted as a morphism of $A$-torsors.
\begin{equation*}
+_{1}\colon  E_{x,y}\wedge^{A}E_{x',y}\to E_{xx',y}
\end{equation*}
\begin{equation*}
+_{2}\colon  E_{x,y}\wedge^{A} E_{x,y'} \to E_{x,yy'}.
\end{equation*}
Here $E_{x,y}$ denotes the fiber above a point $(x,y) \in G\times H$. These two composition laws are required to be 
\begin{enumerate}
    \item associative
    \item compatible with each other, i.e.
\begin{equation*}
(X_{x,y}+_{1}X_{x',y})+_{2}(X_{x,y'}+_{1}X_{x',y'}) = (X_{x,y}+_{1}X_{x,y'})+_{2}(X_{x',y}+_{1}X_{x',y'}).
\end{equation*}
\end{enumerate}
\end{definition}
\begin{remark}
Breen assumes $G$ and $H$ to be abelian groups and $+_1$, $+_2$ to be commutative laws in \cite{Breen}. 
\end{remark}

A biextension $E$ can be trivialized, i.e., it can be defined by $A\times G\times H$ using two partial group laws that satisfy the required conditions. These are cocycle conditions in appropriate cohomology groups. See \cite{Mumford} for the details.

\subsection{Cocycles related to a Biextension}
For chosen elements $X_{x,y} \in E_{x,y}$, we have 
\begin{equation*}
X_{x,y} +_{1} X_{x',y} = A_{x,x';y} + X_{xx',y} 
\end{equation*}
\begin{equation*}
X_{x,y} +_{2} X_{x,y'} = B_{x;y,y'} + X_{x,yy'} 
\end{equation*}
Here $A_{x,x';y}$ and $B_{x;y,y'}$ are elements in $\pi_{1}(\mathscr{E})$, and $+$ is the action of $\pi_{1}(\mathscr{E})$ on $E_{-,-}$. Along with this, $A_{-,-;-}$ and $B_{-;-,-}$ satisfy cocycle conditions in $H^{2}(\pi_{0}\mathscr{E}\times \pi_{0}\mathscr{E}, \pi_{1}\mathscr{E})$ as below. 
\begin{enumerate}
\item Due to associativity of $+_{1}$ and $+_{2}$.
\begin{equation*}
    A_{x,x';y}+A_{xx',x'';y} = A_{x',x'';y} + A_{x,x'x'';y}
\end{equation*}
\begin{equation*}
    B_{x;y,y'} + B_{x;yy',y''} = B_{x;y',y''} + B_{x;y,y'y''}    
\end{equation*}
\item Due to compatibility of $+_1$ and $+_2$.
    \begin{equation*}
A_{x,x';y} + A_{x,x';y'} + B_{xx';y,y'} = B_{x;y,y'} + B_{x';y,y'} + A_{x,x';yy'}.
\end{equation*}
\end{enumerate}
After seeing the cocycles for any general biextension, we focus our attention to the a specific biextension that we can construct from a given monoidal category. 
\subsection{Biextensions and Monoidal Categories}
For a monoidal category $\mathscr{E}$, $\pi_{1}(\mathscr{E})$ can be proved to be an abelian group using the Eckmann-Hilton argument. Moreover, for any $X,Y \in \text{Ob}(\mathscr{E})$, we can show that $\text{Hom}_{\mathscr{E}}(X,Y)$ is a $\pi_{1}(\mathscr{E})$-torsor. We can consider the action $+\colon  \pi_{1}(\mathscr{E})\times \text{Hom}_{\mathscr{E}}(X,Y) \to \text{Hom}_{\mathscr{E}}(X,Y)$ as $(h,f) \mapsto f \circ \eta_{X}(h)$. Similarly, one can also define a right action.

So we consider a biextension of $\pi_{0}(\mathscr{E})\times \pi_{0}(\mathscr{E})$ by $\pi_{1}(\mathscr{E})$, endowed with the following pair of partial composition laws as shown in \cite{Breen}.
\begin{equation*}
+_{1}\colon  {E}_{x,y}\wedge^{\pi_{1}\mathscr{E}}{E}_{x',y}\to {E}_{xx',y} 
\end{equation*}
\begin{equation*}
+_{2}\colon  {E}_{x,y}\wedge^{\pi_{1}\mathscr{E}} {E}_{x,y'} \to {E}_{x,yy'}.
\end{equation*}
Here ${E}_{x,y} = \text{Hom}_{\mathscr{E}}(YX, XY)$ denotes a fiber above $(x,y)\in \pi_{0}(\mathscr{E})\times \pi_{0}(\mathscr{E})$ with chosen $X,Y \in \text{Ob}(\mathscr{E})$ that represent the classes $x,y$, respectively. The partial composition laws are set maps defined as follows. For $f \in {E}_{x,y}$, $f' \in {E}_{x',y}$, and $g' \in {E}_{x,y'}$, we define 
\begin{equation*}
f +_{1} f' \colon =  X\cdot f'\circ f\cdot X' = YXX'\xrightarrow{f\cdot X'} XYX' \xrightarrow{Xf'} XX'Y,
\end{equation*}
\begin{equation*}
f +_{2} g' \colon =  f\cdot Y'\circ Y\cdot g' = YY'X\xrightarrow{Y\cdot g'} YXY' \xrightarrow{f\cdot Y'} XYY'.
\end{equation*}
It can be easily checked that these are well-defined under the action of $\pi_{1}(\mathscr{E})$. These partial composition laws must satisfy conditions like associativity, commutativity, and compatibility with each other, and that gives us corresponding cocycles as shown in \cite[\S 2]{Breen}. 

We now upgrade and define categorical biextension of groups $G$ and $H$ by a Picard groupoid $\cA$.
\section{Categorical Biextensions}\label{biExtMonbiCat}
\begin{definition}
Let $G, H$ be groups, $\cA$ be a Picard category, then a biextension, $\mathscr{E}$, of $G\times H$ by $\cA$ is an $\cA$-torsor over $G\times H$, endowed with a pair of functors of partial composition laws whose restriction to appropriate fibers may be depicted as a morphism of $\cA$-torsors. 
    \begin{equation*}+_{1}\colon  \mathscr{E}_{x,y}\wedge^{\cA}\mathscr{E}_{x',y}\to \mathscr{E}_{xx',y} \end{equation*}
     \begin{equation*} +_{2}\colon  \mathscr{E}_{x,y}\wedge^{\cA} \mathscr{E}_{x,y'} \to \mathscr{E}_{x,yy'}.\end{equation*}
     Here $\mathscr{E}_{x,y}$ denotes the fiber above a point $(x,y) \in G\times H$. These two composition laws are required to be 
    \begin{enumerate}
        \item associative up to a coherent isomorphism
        \item compatible with each other up to a coherent isomorphism
        \begin{equation*}ca\colon  (X_{x,y}+_{1}X_{x',y})+_{2}(X_{x,y'}+_{1}X_{x',y'}) \xrightarrow{\cong} (X_{x,y}+_{1}X_{x,y'})+_{2}(X_{x',y}+_{1}X_{x',y'}).\end{equation*}
    \end{enumerate} 
When $G$ and $H$ are abelian, we can talk about biextensions with full symmetry, and in that case, we let the composition laws be commutative up to a coherent isomorphism.
\end{definition}
The name $ca$ comes from commutativity and associativity as the map can be decomposed in terms of the two. 
\subsection{Cocycles related to a Biextension}\label{CocycleBiext}
For chosen objects $X_{x,y}$ of the category $\mathscr{E}_{x,y}$ we can again write 
\begin{equation*}f_{x,x';y}\colon  X_{x,y} +_{1} X_{x',y} \to A_{x,x';y} + X_{xx',y} \end{equation*}
\begin{equation*}g_{x;y,y'}\colon  X_{x,y} +_{2} X_{x,y'} \to B_{x;y,y'} + X_{x,yy'} \end{equation*}
Here $A_{x,x';y}$ and $B_{x;y,y'}$ are objects of the groupoid $\cA$, and $+$ is the action of $\cA$ on $\mathscr{E}_{-,-}$. $A_{-,-;-}$ and $B_{-;-,-}$ satisfy cocycle conditions in $H^{3}(G\times H, \cA)$ as below. 
\begin{enumerate}
    \item Associativity of $+_1$ and $+_2$. 

Using similar calculations as in theorem \ref{thmExt}, we get the cocycles to be $(f_{x,x',x'';y},\theta_{x,x',x'',x''';y}),\\ (f_{x;y,y',y''},\theta_{x;y,y',y'',y'''})$ in $H^{3}(G\times H),\cA)).$ Here $y$ and $x$ act as spectators in the first and the second cocycle respectively.  
    \item Compatibility of $+_1$ and $+_2$.

\begin{equation*}
\begin{tikzcd}
    & {\mathscr{E}_{xx',yy'}} &                                                                                                               \\
{\mathscr{E}_{xx',y}\wedge \mathscr{E}_{xx',y'}} \arrow[ru] \arrow[rrd, "{\chi_{x,x';y,y'}}", Rightarrow, yshift = 10, xshift = 15]                &                         & {\mathscr{E}_{x,yy'}\wedge \mathscr{E}_{x',yy'}} \arrow[lu]                                                   \\
{(\mathscr{E}_{x,y}\wedge \mathscr{E}_{x',y})\wedge (\mathscr{E}_{x,y'}\wedge \mathscr{E}_{x',y'})} \arrow[rr,"{ca}"'] \arrow[u] &                         & {(\mathscr{E}_{x,y}\wedge \mathscr{E}_{x,y'})\wedge (\mathscr{E}_{x'y,}\wedge \mathscr{E}_{x',y'})} \arrow[u]
\end{tikzcd}
\end{equation*}
In other words, $\chi_{x,x';y,y'}\colon  A_{x,x';y} + A_{x,x';y'} + B_{xx';y,y'} \xrightarrow{} B_{x;y,y'} + B_{x';y,y'} + A_{x,x';yy'} $ is a morphism in $\cA$ that must satisfy the compatibility conditions with 5 entries. 
\begin{enumerate}
    \item (3,2)-coherence axiom

See diagram \ref{(3,2)-diagram}.Assuming other structural morphisms to be trivial, this boils down to the following condition:
\begin{equation*}\chi_{x,x'x'';y,y'}\cdot \chi_{x',x'';y,y'} = \chi_{xx',x'';y,y'} \cdot \chi_{x,x';y,y'}. \end{equation*}
Note that the diagram \ref{(3,2)-diagram} is a polytope and its base commutes as it is made up of structural 2-morphisms of $\cA$.
        \item (2,3)-coherence axiom
        
Similarly, this case boils down to the following condition:
\begin{equation*}\chi_{x,x';y',y''}\cdot \chi_{x,x';y,y'y''} = \chi_{x,x';y,y'}\cdot \chi_{x,x';yy',y''}. \end{equation*}
    \end{enumerate}
\end{enumerate}
\begin{remark}
    If one goes through a similar analysis for the (3,3)-coherence axiom, they will get a diagram in 4-dimensional space whose base, in the 3-dimensions can be described using (2,3) and (3,2)-coherence axioms as above. The base of (3,3)-coherence axiom commutes because of the following diagrams.
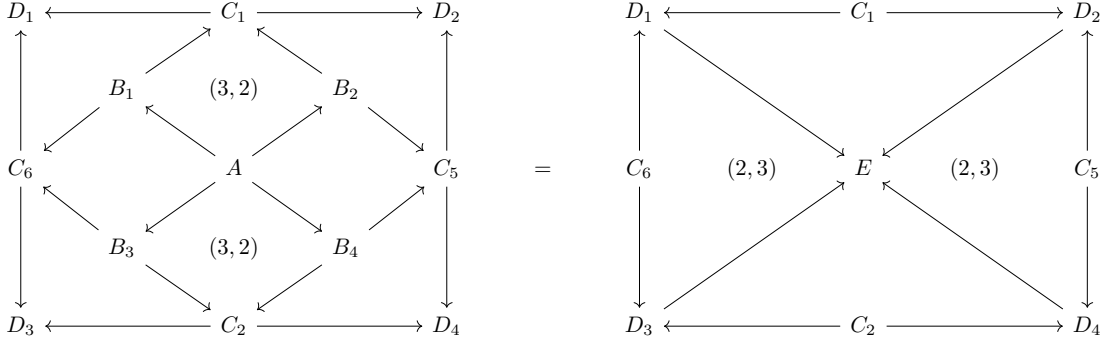
\begin{equation*}
\adjustbox{scale = 0.85}{%
\begin{tikzcd}
D_{1}                       &                             & C_{1} \arrow[rr] \arrow[ll]                   &                             & D_{2}                       &   & D_1 \arrow[rrdd]          &         & C_1 \arrow[ll] \arrow[rr] &         & D_2 \arrow[lldd]          \\
& B_{1} \arrow[ld] \arrow[ru] & {(3,2)}                                       & B_{2} \arrow[rd] \arrow[lu] &                             &   &                           &         &                           &         &                           \\
C_{6} \arrow[dd] \arrow[uu] &                             & A \arrow[ru] \arrow[rd] \arrow[ld] \arrow[lu] &                             & C_{5} \arrow[dd] \arrow[uu] & = & C_6 \arrow[uu] \arrow[dd] & {(2,3)} & E                         & {(2,3)} & C_5 \arrow[uu] \arrow[dd] \\
& B_{3} \arrow[rd] \arrow[lu] & {(3,2)}                                       & B_{4} \arrow[ru] \arrow[ld] &                             &   &                           &         &                           &         &                           \\
D_{3}                       &                             & C_{2} \arrow[ll] \arrow[rr]                   &                             & D_{4}                       &   & D_3 \arrow[rruu]          &         & C_2 \arrow[rr] \arrow[ll] &         & D_4 \arrow[lluu]         
\end{tikzcd}} 
\end{equation*}

Details of the nodes of these diagrams can be found in diagrams \ref{nodes(3,3)}. We do not need the (3,3)-coherence axiom for this analysis as (3,2) and (2,3)-coherence axioms are sufficient.
\end{remark}

    \subsection{Biextension and Monoidal bi-Category}\label{biextMonbiCat}
        \begin{lemma}
    \begin{enumerate}
    \item $\cAut_{\mathbb{E}}(X)$ is a Picard groupoid.
    \item For each $X \in \text{Ob}(\mathbb{E})$, $\eta_{X}$ is a monoidal functor of Picard groupoids.
    \end{enumerate}
\end{lemma}
\begin{proof}
    Since $\cA$ and $\cAut_{\mathbb{E}}(X)$ are equivalent categories, $\cAut_{\mathbb{E}}(X)$ is also a Picard groupoid. Moreover, the following coherence condition for the right unitor $r_{X}$ makes $\eta_{X}$ a monoidal functor:

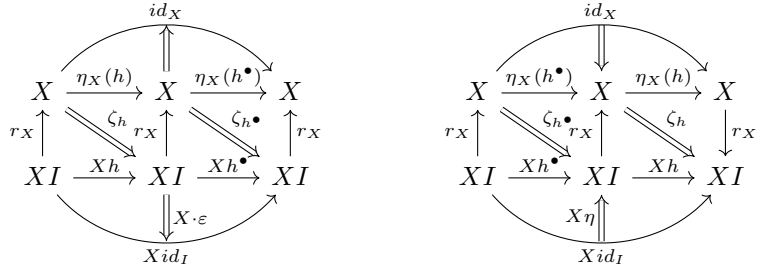
\begin{equation*}
\begin{tikzcd}
& {}                                                                                                       &                        &  &                                                                                                                       & {} \arrow[d, Rightarrow]                                       &                      \\
X \arrow[r, "\eta_{X}(h)"] \arrow[rd, "\zeta_{h}", Rightarrow] \arrow[rr, "id_{X}", bend left=49] & X \arrow[r, "\eta_{X}(h^{\bullet})"] \arrow[rd, "\zeta_{h^{\bullet}}", Rightarrow] \arrow[u, Rightarrow] & X                      &  & X \arrow[r, "\eta_{X}(h^{\bullet})"] \arrow[rr, "id_{X}", bend left=49] \arrow[rd, "\zeta_{h^{\bullet}}", Rightarrow] & X \arrow[r, "\eta_{X}(h)"] \arrow[rd, "\zeta_{h}", Rightarrow] & X \arrow[d, "r_{X}"] \\
XI \arrow[u, "r_{X}"] \arrow[r, "Xh"] \arrow[rr, "Xid_{I}"', bend right=49]                       & XI \arrow[r, "Xh^{\bullet}"] \arrow[u, "r_{X}"] \arrow[d, "X\cdot\varepsilon", Rightarrow]               & XI \arrow[u, "r_{X}"'] &  & XI \arrow[u, "r_{X}"] \arrow[rr, "Xid_{I}"', bend right=49] \arrow[r, "Xh^{\bullet}"]                                 & XI \arrow[u, "r_{X}"] \arrow[r, "Xh"]                          & XI                   \\
& {}                                                                                                       &                        &  &                                                                                                                       & {} \arrow[u, "X\eta", Rightarrow]                              &                     
\end{tikzcd}
\end{equation*}
\end{proof}

The same results can be generalized further to any Hom-sets.
\begin{lemma}
    \begin{enumerate}
        \item Let $\mathbb{E}$ be a monoidal bi-category, then for any $X, Y \in \text{Ob}(\mathbb{E})$, $\mathscr{H}$om$_{\mathbb{E}}(X,Y)$ is an $\cAut_{\mathbb{E}}(I)$-torsor.
        \item Again, since $\cA$ is a Picard groupoid, left and right actions by $\cA$ are equivalent. 
    \end{enumerate}
\end{lemma}
\begin{proof}
    Again we can consider the action $+\colon  \cA\times \mathscr{H}\text{om}_{\mathbb{E}}(X,Y) \to \mathscr{H}\text{om}_{\mathbb{E}}(X,Y)$ as $(h,f) \mapsto f \circ \eta_{X}(h)$.
\end{proof}
Let $\mathscr{E}_{x,y} \colon = \mathscr{H}$om$_{\mathbb{E}}(YX,XY)$ for $x,y \in \pi_{0}(\mathbb{E})$, and for chosen $X,Y \in \text{Ob}(\mathbb{E})$ that represent the classes $x,y$, respectively. Now we define the functors:
\begin{equation*}+_{1}\colon  \mathscr{E}_{x,y}\wedge^{\cA}\mathscr{E}_{x',y}\to \mathscr{E}_{xx',y},\ \ \ +_{2}\colon  \mathscr{E}_{x,y}\wedge^{\cA}\mathscr{E}_{x,y'}\to \mathscr{E}_{x,yy'}.\end{equation*}

For $+_{1}\colon$  $(f,f')\in \text{Ob}(\mathscr{E}_{x,y})\times \text{Ob}(\mathscr{E}_{x',y}) \mapsto X\cdot f'\circ f\cdot X' = YXX'\xrightarrow{f\cdot X'} XYX' \xrightarrow{Xf'} XX'Y $
A map $[(\alpha,h,\alpha')]$ maps to the following morphism in $\mathscr{H}$om$_{\mathbb{E}}(YXX',XX'Y)$, i.e., the following 2-morphism in $\mathbb{E}$:
\begin{equation*}
\adjustbox{scale = 0.8, center}{%
\begin{tikzcd}
Y(XX') \arrow[d, "\cong"']                             &                                       &                           &                                                            &                                      &                                                                &  & (XX')Y                                                                                    \\
(YX)X' \arrow[d, "\eta_{YX}(h)X'"'] \arrow[rrd, "fX'"] & {} \arrow[ld, "\alpha X", Rightarrow, shorten <= 2.0em, shorten >= 0em] &                           &                                                            &                                      & {}                                                             &  & X(X'Y) \arrow[u, "\cong"]                                                                 \\
(YX)X' \arrow[rr, "gX'"']                              &                                       & (XY)X' \arrow[r, "\cong"] & X(YX') \arrow[rr, "Xf'"] \arrow[rr, "Xg'"', bend right=60] & {} \arrow[d, "X\alpha'", Rightarrow] & X(X'Y) \arrow[rr, "\eta_{X'Y}(h)"'] \arrow[rru, "id_{X(X'Y)}"] &  & X(X'Y) \arrow[u, "X\eta_{X'Y}(h^{\bullet})"'] \arrow[llu, "X\kappa_{X'Y}(h)", Rightarrow, shorten <= 2.3em, shorten >= 2.5em, xshift = 40, yshift = -7] \\
&                                       &                           &                                                            & {}                                   &                                                                &  &                                                                                          
\end{tikzcd}}
\end{equation*}
Now $+_{1}$ and similarly defined $+_{2}$ must satisfy the following conditions.
\begin{enumerate}
    \item Existence of associativity functor for both, $+_{1}$ and $+_{2}$ with a natural transformation satisfying the associahedron $K_{5}$.
    \item Existence of a functor $ca\colon  (-+_{1}-)+_{2}(-+_{1}-) \to (-+_{2}-)+_{1}(-+_{2}-)$ with a natural transformation $\chi_{-,-,-,-}$.
    \item Coherence of $\chi_{-,-,-,-}\colon$ 
        \begin{enumerate}
            \item (3,2)-coherence axiom.
            \item (2,3)-coherence axiom.     
        \end{enumerate}
\end{enumerate}
These give rise to cocycles with value in the Picard groupoid $\cA$ as described in detail above.

\section{Fully Symmetric case}\label{fullSym}
    \subsection{Biextensions and Symmetric Monoidal Categories}\label{biExtSymMonCat}
    If we start with a symmetric monoidal category $\mathscr{E}$, we can again consider a biextension of $(\pi_{0}\mathscr{E}\times \pi_{0}\mathscr{E})$ by $\pi_{1}(\mathscr{E})$ as in section \ref{biExtMonCat}. The symmetry condition gives us two obvious cocycles as follows.
    \subsubsection{Cocycles related to the Biextension} 
    Symmetry implies $X_{x,y} +_1 X_{x',y} = X_{x',y}+_1 X_{x,y}$, i.e, \begin{equation*}A_{x,x';y} = A_{x',x;y}\end{equation*} because $xx' = x'x \in \pi_{0}(\mathscr{E})$ due to the presence of the braiding isomorphism between $XX'$ and $X'X$. Similarly, for $+_2$, we get the following cocycle condition:
    \begin{equation*}B_{x;y,y'} = B_{x;y',y}.\end{equation*}
    We still have the other cocycle conditions due to associativity and compatibility as described in section \ref{biExtMonCat}.
    
    \subsection{Biextensions and Symmetric Monoidal bi-Categories}\label{biExtSymMonbiCat}
    For a given symmetric monoidal bi-category $\mathbb{E}$, we can once again consider the biextension as described in section \ref{biextMonbiCat}. Now the two partial composition laws must satisfy the commutativity condition along with associativity and compatibility of $+_1$ and $+_2$. This gives rise to some more cocycle conditions including the existing ones as described earlier.
    \subsubsection{Cocycles related to the Biextension} 
    For chosen objects $X_{x,y}$ of the category $\mathscr{E}_{x,y}$ we can again write 
    \begin{equation*}f_{x,x';y}\colon  X_{x,y} +_{1} X_{x',y} \to A_{x,x';y} + X_{xx',y} \end{equation*}
    \begin{equation*}g_{x;y,y'}\colon  X_{x,y} +_{2} X_{x,y'} \to B_{x;y,y'} + X_{x,yy'} \end{equation*}

    In the presence of symmetry, we get the following extra structure.

    \begin{equation*}
\begin{tikzcd}
{\mathscr{E}_{x,y}\wedge \mathscr{E}_{x',y}} \arrow[d] \arrow[r]                                & {\mathscr{E}_{xx',y}} \arrow[d, "id"] &  & {\mathscr{E}_{x,y}\wedge \mathscr{E}_{x,y'}} \arrow[d] \arrow[r]                                & {\mathscr{E}_{x,yy'}} \arrow[d, "id"] \\
{\mathscr{E}_{x',y}\wedge \mathscr{E}_{x,y}} \arrow[r] \arrow[ru, "{\mu_{x,x';y}}", Rightarrow] & {\mathscr{E}_{xx',y}}                 &  & {\mathscr{E}_{x,y'}\wedge \mathscr{E}_{x,y}} \arrow[r] \arrow[ru, "{\mu_{x;y,y'}}", Rightarrow] & {\mathscr{E}_{x,yy'}}                
\end{tikzcd}
\end{equation*}
    In other words, we get the morphisms in $\cA$ as follows:
    \begin{equation*}\mu_{x,x';y}\colon A_{x,x';y} \to A_{x',x;y} \end{equation*}
    \begin{equation*}\mu_{x;y,y'}\colon B_{x;y,y'} \to B_{x;y',y} \end{equation*}
    These must satisfy the following cocycle conditions.

    \begin{enumerate}
        \item Interaction with syllepsis $\gamma$.

        \begin{equation*}
\begin{tikzcd}
    & {\mathscr{E}_{x,y}\wedge \mathscr{E}_{x',y}} \arrow[d,"{\tau}"] \arrow[r] \arrow[dd, bend right=65,"{id}"']                                          & {\mathscr{E}_{xx',y}} \arrow[d, "id"] &                                               & {\mathscr{E}_{x,y}\wedge \mathscr{E}_{x,y'}} \arrow[d,"{\tau}"] \arrow[r] \arrow[dd, bend right=65,"{id}"']                                          & {\mathscr{E}_{x,yy'}} \arrow[d, "id"] \\
{} \arrow[r, "{\gamma_{x,x';y}}", Rightarrow, xshift = 28, yshift = 13, shorten <= 0.2em, shorten >= 0em] & {\mathscr{E}_{x',y}\wedge \mathscr{E}_{x,y}} \arrow[d,"{\tau}"] \arrow[r] \arrow[dd, bend right=65,"{id}"'] \arrow[ru, "{\mu_{x,x';y}}", Rightarrow] & {\mathscr{E}_{xx',y}} \arrow[d, "id"] & {} \arrow[r, "{\gamma_{x;y,y'}}", Rightarrow, Rightarrow, xshift = 28, yshift = 13, shorten <= 0.2em, shorten >= 0em] & {\mathscr{E}_{x,y'}\wedge \mathscr{E}_{x,y}} \arrow[d,"{\tau}"] \arrow[r] \arrow[dd, bend right=65,"{id}"'] \arrow[ru, "{\mu_{x;y,y'}}", Rightarrow] & {\mathscr{E}_{x,yy'}} \arrow[d, "id"] \\
{} \arrow[r, "{\gamma_{x',x;y}}", Rightarrow, Rightarrow, xshift = 28, yshift = -13, shorten <= 0.2em, shorten >= 0em] & {\mathscr{E}_{x,y}\wedge \mathscr{E}_{x',y}} \arrow[r] \arrow[d,"{\tau}"] \arrow[ru, "{\mu_{x',x;y}}", Rightarrow]                           & {\mathscr{E}_{xx',y}} \arrow[d, "id"] & {} \arrow[r, "\gamma_{x;y'y}", Rightarrow, Rightarrow, xshift = 28, yshift = -13, shorten <= 0.2em, shorten >= 0em]    & {\mathscr{E}_{x,y}\wedge \mathscr{E}_{x,y'}} \arrow[d,"{\tau}"] \arrow[r] \arrow[ru, "{\mu_{x;y',y}}", Rightarrow]                           & {\mathscr{E}_{x,yy'}} \arrow[d, "id"] \\
    & {\mathscr{E}_{x',y}\wedge \mathscr{E}_{x,y}} \arrow[r] \arrow[ru, "{\mu_{x,x';y}}", Rightarrow]                                     & {\mathscr{E}_{xx',y}}                 &                                               & {\mathscr{E}_{x,y'}\wedge \mathscr{E}_{x,y}} \arrow[r] \arrow[ru, "{\mu_{x;y,y'}}", Rightarrow]                                     & {\mathscr{E}_{x,yy'}}                
\end{tikzcd}
\end{equation*}
        The syllepsis $\gamma$ can also be considered as a morphism of $\cA$ as follows:
        \begin{equation*}\gamma_{x,x';y}\colon  A_{x,x';y} \to A_{x,x';y} \end{equation*}
        \begin{equation*}\gamma_{x;y,y'}\colon  B_{x;y,y'} \to B_{x;y,y'} \end{equation*}
        So from the diagram above, we get the following two cocycle conditions.
        \begin{equation*}\mu_{x,x';y}\cdot \gamma_{x',x;y} = \gamma_{x,x';y}\cdot \mu_{x,x';y} \end{equation*}
        \begin{equation*}\mu_{x;y,y'}\cdot \gamma_{x;y',y} = \gamma_{x;y,y'}\cdot \mu_{x;y,y'} \end{equation*}
        \item \begin{enumerate}
            \item (3,1)-coherence axiom for $\mu.$ This arises from the interaction of $\mu$ with the braiding.

        \begin{equation*}
\begin{tikzcd}
{\mathscr{E}_{xx'x'',y}} \arrow[rrr,"id"]             &                                                                                                                                              &                                                                                                                                                                      & {\mathscr{E}_{xx'x'',y}}             \\
& {(\mathscr{E}_{x,y}\wedge\mathscr{E}_{x',y})\wedge \mathscr{E}_{x'',y}} \arrow[d, "\tau\wedge id"] \arrow[r, "\tau"] \arrow[lu]                 & {\mathscr{E}_{x'',y}\wedge(\mathscr{E}_{x,y}\wedge \mathscr{E}_{x',y})} \arrow[d, "id\wedge \tau"'] \arrow[ru] \arrow[llu, "{\mu_{xx',x'';y}}"', Rightarrow, shorten <= 0.0em, shorten >= 0.5em, xshift = 30]               &                                   \\
& {(\mathscr{E}_{x',y}\wedge \mathscr{E}_{x,y})\wedge \mathscr{E}_{x'',y}} \arrow[r, "\tau"'] \arrow[ld] \arrow[luu, "{\mu_{x,x';y}}", Rightarrow,yshift = -5] & {\mathscr{E}_{x'',y}\wedge (\mathscr{E}_{x',y}\wedge \mathscr{E}_{x,y})} \arrow[rd] \arrow[lld, "{\mu_{xx',x'';y}}", Rightarrow, shorten <= 0.0em, shorten >= 0.5em, xshift = 30] \arrow[ruu, "{\mu_{x,x';y}}"', Rightarrow, yshift = -5] &                                   \\
{\mathscr{E}_{xx'x'',y}} \arrow[uuu,"id"] \arrow[rrr,"id"'] &                                                                                                                                              &                                                                                                                                                                      & {\mathscr{E}_{xx'x'',y}} \arrow[uuu,"id"']
\end{tikzcd}
\end{equation*}
This gives us the following cocycle condition.
\begin{equation*}\mu_{xx',x'';y}\cdot \mu_{x,x';y} = \mu_{x,x';y}\cdot \mu_{xx',x'';y} \end{equation*}

\item (3,1)-coherence axiom for $\mu.$

\begin{equation*}
\begin{tikzcd}
{\mathscr{E}_{x,yy'y''}} \arrow[rrr, "id"]                    &                                                                                                                                                  &                                                                                                                                                                             & {\mathscr{E}_{x,yy'y''}}                    \\
& {(\mathscr{E}_{x,y}\wedge\mathscr{E}_{x,y'})\wedge \mathscr{E}_{x,y''}} \arrow[d, "\tau\wedge id"] \arrow[r, "\tau"] \arrow[lu]                  & {\mathscr{E}_{x,y''}\wedge(\mathscr{E}_{x,y}\wedge \mathscr{E}_{x,y'})} \arrow[d, "id\wedge \tau"'] \arrow[ru] \arrow[llu, "{\mu_{x;yy',y''}}"', Rightarrow, shorten <= 0.0em, shorten >= 0.5em, xshift = 30]                &                                             \\
& {(\mathscr{E}_{x,y'}\wedge \mathscr{E}_{x,y})\wedge \mathscr{E}_{x,y''}} \arrow[r, "\tau"'] \arrow[ld] \arrow[luu, "{\mu_{x;y,y'}}", Rightarrow, yshift = -5] & {\mathscr{E}_{x,y''}\wedge (\mathscr{E}_{x,y'}\wedge \mathscr{E}_{x,y})} \arrow[rd] \arrow[lld, "{\mu_{x;yy',y''}}", Rightarrow, shorten <= 0.0em, shorten >= 0.5em, xshift = 30] \arrow[ruu, "{\mu_{x;y,y'}}"', Rightarrow, yshift = -5] &                                             \\
{\mathscr{E}_{x,yy'y''}} \arrow[uuu, "id"] \arrow[rrr, "id"'] &                                                                                                                                                  &                                                                                                                                                                             & {\mathscr{E}_{x,yy'y''}} \arrow[uuu, "id"']
\end{tikzcd}
\end{equation*}
This gives us the following cocycle condition.
\begin{equation*}\mu_{x;yy',y''}\cdot \mu_{x;y,y'} = \mu_{x;y,y'}\cdot \mu_{x;yy',y''} \end{equation*}
\end{enumerate}
        \item (2,2)-coherence axiom for $\mu$. Interaction between two $\mu$'s via $\chi$.
See diagram \ref{diagram3}. This gives us the following condition:
\begin{equation*}\mu_{xx';y,y'} \cdot  (\mu_{x,x';y} + \mu_{x,x';y'}) \cdot \chi_{x,x';y,y'}  = \chi_{x',x;y',y}\cdot \mu_{x,x';yy'} \cdot (\mu_{x;y,y'} + \mu_{x';y,y'}).\end{equation*}
    \end{enumerate}
    \subsection{Biextensions and MacLane Cohomology}\label{HML}  
    For any group $G$, and a Picard category $\cA$, we can construct the following diagram of cohomology groups. Here the horizontal maps are a part of long exact sequences that can arise from the short exact sequence $0\to \Sigma(\pi_{1}\cA)\to \cA\to \pi_{0}\cA\to 0$, and the vertical maps arise from the suspension on the first entry. As we move sufficiently down, we enter the stable range. After getting in the stable range for cohomology with values in a group, we can chase the diagram to verify that the sequence of stable cohomology with values in a Picard category is also exact.  

    \begin{equation*}
\begin{tikzcd}
\cdots \arrow[r] & {H^3(K(G,1),\pi_1\cA)} \arrow[r]                  & {H^2(K(G,1),\cA)} \arrow[r]           & {H^2(K(G,1),\pi_0\cA)} \arrow[r]                     & \cdots \\
\cdots \arrow[r] & {H^4(K(G,2),\pi_1\cA)} \arrow[u] \arrow[r]        & {H^3(K(G,2),\cA)} \arrow[u] \arrow[r] & {H^3(K(G,2),\pi_0\cA)} \arrow[r] \arrow[u]           & \cdots \\
\cdots \arrow[r] & {H^5(K(G,3),\pi_1\cA)} \arrow[u] \arrow[r]        & {H^4(K(G,3),\cA)} \arrow[u] \arrow[r] & {H^4(K(G,3),\pi_0\cA)} \arrow[r] \arrow[u, "\cong"'] & \cdots \\
& {H^2_{st}(G,\pi_1\cA)} \arrow[u, "\cong"] \arrow[r] & H^{1}_{st}(G,\cA) \arrow[r]                                              & {H^1_{st}(G,\pi_0\cA)} \arrow[u, "\cong"']             &       
\end{tikzcd}
\end{equation*}
As shown by Eilenberg and MacLane in \cite{HML}, the stable cohomology of Eilenberg MacLane spaces is isomorphic to the cohomology of the $Q$-complex. 
\subsubsection{$Q$-complex for a symmetric monoidal category}\label{QSMCat}
For a given symmetric monoidal category $\mathscr{E}$, we can consider the extension $0\to \Sigma (\pi_{1}(\mathscr{E}))\to \mathscr{E}\to \pi_0(\mathscr{E})\to 0$ as earlier. 
Let \adjustbox{scale = 1}{%
\begin{tikzcd}
x \arrow[r, no head] & y \arrow[r, no head] & z
\end{tikzcd}
    } denote an element in $Q_{1}(\pi_0\mathscr{E})$, i.e., $y = x+z$. Note that, since $\mathscr{E}$ is braided, $x+z = z+x$ for every $x,z \in \pi_0\mathbb{E}$. Similarly, let 
    \begin{tikzcd}
x \arrow[r, no head] & y \arrow[r, no head] & z
\end{tikzcd} an element in $Q_{1}(\cA)$ if there is an isomorphism $x+z\xrightarrow{\cong} y$.
Now let us consider a section $s\colon  \pi_0\mathscr{E} \to \mathscr{E}$. Due to the monoidal structure of $\mathscr{E}$, for each $x,y \in \pi_{0}\mathscr{E}$ and the fact that $s(-)$ is an $\pi_1$-torsor, we have a morphism of $\pi_1$-torsors, $\lambda_{x,y}\colon s(x) \wedge^{\pi_1} s(y) \xrightarrow{\cong} s(x+y)$. This gives us elements of the form \begin{tikzcd}
s(x) \arrow[r, no head] & s(x+y) \arrow[r, no head] & s(y)
\end{tikzcd} in $Q_{1}(\cA)$. Moreover, these must satisfy the following cocycle conditions, i.e., the element mentioned below must lie in $Q_{2}(\cA)$.

\begin{equation*}
\begin{tikzcd}
s(x) \arrow[r, no head] \arrow[d, no head]              & {s(x+y)} \arrow[r, no head] \arrow[d, no head] & s(y) \arrow[d, no head]              \\
{s(x+z) } \arrow[r, no head] \arrow[d, no head] & \bullet \arrow[r, no head] \arrow[d, no head]           & {s(y+t)} \arrow[d, no head] \\
s(z) \arrow[r, no head]                                 & {s(z+t)} \arrow[r, no head]                    & s(t)                                
\end{tikzcd}
\end{equation*}
In other words, the existence of this cube implies that $s(x+y) + s(z+t) \cong s(x+z) + s(y+t)$. Notice that this single condition represents associativity as well as commutativity. If we make $z = 0$, the diagram says $s(x+y) + s(t) \cong s(x)+s(y+t)$, similarly, making $x,t = 0$, we get $s(y) + s(z) = s(z) + s(y)$. 

\subsubsection{$Q$-complex for a symmetric monoidal bi-category}
In the case of a symmetric monoidal bi-category $\mathbb{E}$, in addition to the previous structure, for each $x,y \in \pi_0\mathbb{E}$, we have $\lambda_{x,y}\colon s(x) + s(y) \xrightarrow{\cong} s(x+y) + c(x,y)$ for a $c(x,y)\in \text{Ob}(\pi_1\mathbb{E})$. Hence, in this case, we have a non-trivial element of $Q_{2}(\cA)$

\begin{equation*}
\begin{tikzcd}
s(x) \arrow[r, no head] \arrow[d, no head]              & {s(x+y) + c(x,y)} \arrow[r, no head] \arrow[d, no head]                                       & s(y) \arrow[d, no head]              \\
{s(x+z) + c(x,z)} \arrow[r, no head] \arrow[d, no head] & s(x+y+z+t)+\theta \begin{pmatrix} x&y\\z&t\end{pmatrix} \arrow[r, no head] \arrow[d, no head] & {s(y+t) + c(y,t)} \arrow[d, no head] \\
s(z) \arrow[r, no head]                                 & {s(z+t) + c(z,t)} \arrow[r, no head]                                                          & s(t)                                
\end{tikzcd}
\end{equation*}
Here $\theta$
$\begin{pmatrix}
    x & y\\
    z & t
\end{pmatrix}\colon$  $c(x+y,z+t) + c(x,y) + c(z,t) \to c(x+z,y+t) + c(x,z) + c(y,t)$ is a morphism in $\cA$ representing the pentagon for associator in $\mathbb{E}$. This must satisfy the following (4,4)-coherence axiom for extensions.
\begin{equation*}
\adjustbox{scale= 0.7}{%
\begin{tikzcd}
((s(x)+s(y))+(s(z)+s(t)))+((s(a)+s(b))+(s(c)+s(d))) \arrow[r, "ca"] \arrow[dd, "ca+ca"'] & ((s(x)+s(y))+(s(a)+s(b)))+((s(z)+s(t))+(s(c)+s(d))) \arrow[dd, "ca+ca"] \\
&                                                                         \\
((s(x)+s(z))+(s(y)+s(t)))+((s(a)+s(c))+(s(b)+s(d))) \arrow[dd, "ca"']                    & ((s(x)+s(a))+(s(y)+s(b)))+((s(z)+s(c))+(s(t)+s(d))) \arrow[dd, "ca"]    \\
&                                                                         \\
((s(x)+s(z))+(s(a)+s(c)))+((s(y)+s(t))+(s(b)+s(d))) \arrow[r, "ca+ca"']                  & ((s(x)+s(a))+(s(z)+s(c)))+((s(y)+s(b))+(s(t)+s(d)))                    
\end{tikzcd}}
\end{equation*}
After expanding this diagram using the monoidal product, we get the following condition for $\theta$.

\begin{equation*}
\begin{multlined}
\theta\begin{pmatrix}
    x+y & z+t\\ a+b & c+d
\end{pmatrix} \cdot \left(\theta\begin{pmatrix}
    x & y\\
    a & b
\end{pmatrix} +\theta\begin{pmatrix}
    z & t\\
    c & d
\end{pmatrix}\right) \cdot \theta \begin{pmatrix}
    x+a & y+b\\
    z+c & t+d
\end{pmatrix} = \\ \left(\theta\begin{pmatrix}
    x & y\\
    z & t
\end{pmatrix} +\theta\begin{pmatrix}
    a & b\\
    c & d
\end{pmatrix}\right) \cdot \theta \begin{pmatrix}
    x+z & y+t\\
    a+c & b+d
\end{pmatrix} \cdot \left(\theta\begin{pmatrix}
    x & z\\
    a & c
\end{pmatrix} +\theta\begin{pmatrix}
    y & t\\
    b & d
\end{pmatrix}\right)
\end{multlined}
\end{equation*}
To analyze the commutativity, we need to consider the compatibility with 16 variables, instead of just the 8 variables. But this becomes too difficult to compute. So we instead use the same technique with biextensions. We consider a biextension for a given monoidal bi-category and put the $Q$-complex structure over $\cA, \mathbb{E}$, and $\pi_0\mathbb{E} \times \pi_0\mathbb{E}$, we get the following two partial composition laws as earlier.
\begin{equation*}X_{x,y} +_1 X_{x',y} \xrightarrow{\cong} X_{x+x',y} +f(x,x';y)\end{equation*}
\begin{equation*}X_{x,y} +_2 X_{x,y'} \xrightarrow{\cong} X_{x,y+y'} + g(x;y,y')\end{equation*} 
Here $X_{-,-}\colon \pi_0\mathbb{E}\times\pi_0\mathbb{E} \to \mathbb{E}$ is a section, and $f(-,-;-), g(-;-,-)\in \text{Ob}(\cA)$. Let us denote the maps $ca_{1}$ and $ca_{2}$ that arise by commutativity and associativity on $+_1$ and $+_2$ individually as follows.
\begin{equation*}ca_{1}\colon  (X_{x,y} +_1 X_{x',y})+_1 (X_{x'',y}+_1 X_{x''',y}) \to (X_{x,y} +_1 X_{x'',y})+_1 (X_{x',y}+_1 X_{x''',y}) \end{equation*}
\begin{equation*}ca_{2}\colon  (X_{x,y} +_2 X_{x,y'})+_2 (X_{x,y''}+_2 X_{x,y'''}) \to (X_{x,y} +_2 X_{x,y''})+_2 (X_{x,y'}+_2 X_{x,y'''}) \end{equation*}

Similarly, we denote the morphism due to the compatibility of $+_1$ and $+_2$ by $ca_{12}$ as follows.
\begin{equation*}ca_{12}\colon  (X_{x,y} +_1 X_{x',y})+_2 (X_{x,y'}+_1 X_{x',y'}) \to (X_{x,y} +_2 X_{x,y'})+_1 (X_{x',y}+_2 X_{x',y'}) \end{equation*}
Now we can break the problem of considering 16 variables into two parts—the (4,2)-coherence axiom and the (2,4)-coherence axiom. The maps $ca_1, ca_{12}$ must satisfy the following (4,2)-coherence axiom (diagram \ref{4,2-coherence}). After expanding this diagram using the partial composition laws we get the following cocycle conditions.

\begin{equation*}
\begin{multlined}
\chi\begin{pmatrix}
    x+x' & x''+x''' \\
    y & y'
\end{pmatrix} \cdot \left(\chi \begin{pmatrix}
    x & x' \\
    y & y'
\end{pmatrix} + \chi \begin{pmatrix}
    x'' & x'''\\
    y & y'
\end{pmatrix}\right) \cdot \theta\left(
    \begin{pmatrix}
        x & x' \\
        x'' & x'''
    \end{pmatrix} ; y+y' \right) = \\
  \left(\theta\left(
    \begin{pmatrix}
        x & x' \\
        x'' & x'''
    \end{pmatrix} ; y \right)+ \theta\left(
    \begin{pmatrix}
        x & x' \\
        x'' & x'''
    \end{pmatrix} ; y' \right) \right)\cdot \chi\begin{pmatrix}
        x+x'' & x'+x'''\\
        y & y'
    \end{pmatrix}\cdot \left(\chi \begin{pmatrix}
    x & x'' \\
    y & y'
\end{pmatrix} + \chi \begin{pmatrix}
    x' & x'''\\
    y & y'
\end{pmatrix}\right)
\end{multlined}
\end{equation*}

Here $\theta$ and $\chi$ are defined as follows.
\begin{equation*}
  \begin{multlined}
    \chi
    \begin{pmatrix}
      x & x' \\
      y & y'
    \end{pmatrix}\colon
    g(x+x';y,y') + f(x,x';y) + f(x,x';y')\lto \\
    f(x,x';y+y') + g(x;y,y') + g(x';y,y'),
  \end{multlined}
\end{equation*}
whereas $\theta \left(\begin{pmatrix}
    x & x' \\
    x'' & x'''
\end{pmatrix}; y \right)$ is the following morphism in $\cA$:
\begin{equation*}
\begin{split}
  f(x+x',x''+x''';y) + f(x,x';y) &+ f(x'',x''';y) \lto \\
  f(x+x'', x'+x''';y) &+ f(x,x'';y) + f(x',x''';y).
\end{split}
\end{equation*}

Similarly, we also need the (2,4)-coherence axiom for $ca_2, ca_{12}$ where we consider $x,x'$ and $y,y',y'',y'''$. One can check that by making $x = 0$ in the (4,2)-coherence axiom, we get back the (3,2)-coherence axiom that we gave in section \ref{CocycleBiext}. Similarly, for the (2,3)-coherence axiom from the (4,2)-coherence axiom. We already saw in section \ref{QSMCat} that we get back associativity and commutativity by restricting some entries of $\theta$ to 0. So overall, the cocycles using $Q$-complex capture all the information in the case of full symmetry.
\begin{theorem}\label{FinalThm}
Let $\mathbb{E}$, $\cA$ be as above. The $\cA$-torsor $\mathscr{E}$ can be represented explicitly by cocyles on $\pi_0\mathbb{E}\times \pi_0\mathbb{E}$ with values in $\cA$, and if $\mathbb{E}$ is symmetric, $\mathscr{E}$ is an alternating biextension. 
\end{theorem}

\newpage
\appendix

\section{Diagrams}\label{diagrams}
\begin{enumerate}
\begin{multicols}{2}
\item (3,2)-coherence axiom\label{(3,2)-diagram}

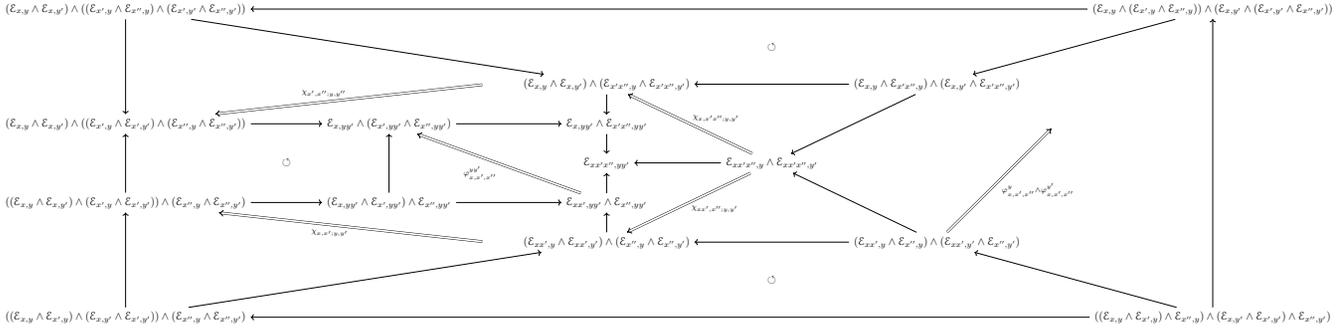
\adjustbox{scale = 0.42, angle=-90,right = 1.7in}{%
\begin{tikzcd}
{(\mathscr{E}_{x,y}\wedge \mathscr{E}_{x,y'})\wedge ((\mathscr{E}_{x',y}\wedge \mathscr{E}_{x'',y})\wedge (\mathscr{E}_{x',y'}\wedge \mathscr{E}_{x'',y'}))} \arrow[rrrrdd] \arrow[ddd] &                  &                                                                                                &                                                  &                                                                                                              &                                                                                                                                                      &                                                                                                                                                                                                       &    & {(\mathscr{E}_{x,y}\wedge (\mathscr{E}_{x',y}\wedge \mathscr{E}_{x'',y}))\wedge (\mathscr{E}_{x,y'}\wedge (\mathscr{E}_{x',y'}\wedge \mathscr{E}_{x'',y'}))} \arrow[lldd] \arrow[llllllll]                  \\
&                  &                                                                                                &                                                  &                                                                                                              & \circlearrowleft                                                                                                                                     &                                                                                                                                                                                                       &    &                                                                                                                                                                                                    \\
&                  &                                                                                                & {} \arrow[llld, "{\chi_{x',x'';y,y''}}"', Rightarrow] & {(\mathscr{E}_{x,y}\wedge \mathscr{E}_{x,y'})\wedge (\mathscr{E}_{x'x'',y}\wedge \mathscr{E}_{x'x'',y'})} \arrow[d]  &                                                                                                                                                      & {(\mathscr{E}_{x,y}\wedge \mathscr{E}_{x'x'',y})\wedge (\mathscr{E}_{x,y'}\wedge \mathscr{E}_{x'x'',y'})} \arrow[ldd] \arrow[ll]                                                                              &    &                                                                                                                                                                                                    \\
{(\mathscr{E}_{x,y}\wedge \mathscr{E}_{x,y'})\wedge ((\mathscr{E}_{x',y}\wedge \mathscr{E}_{x',y'})\wedge (\mathscr{E}_{x'',y}\wedge \mathscr{E}_{x'',y'}))} \arrow[rr]                 &                  & {\mathscr{E}_{x,yy'}\wedge (\mathscr{E}_{x',yy'}\wedge \mathscr{E}_{x'',yy'})} \arrow[rr]            &                                                  & {\mathscr{E}_{x,yy'}\wedge \mathscr{E}_{x'x'',yy'}} \arrow[d]                                                     &                                                                                                                                                      &                                                                                                                                                                                                       & {} &                                                                                                                                                                                                    \\
& \circlearrowleft &                                                                                                &                                                  & {\mathscr{E}_{xx'x'',yy'}}                                                                                       & {\mathscr{E}_{xx'x'',y}\wedge \mathscr{E}_{xx'x'',y'}} \arrow[l] \arrow[ldd, "{\chi_{xx',x'';y,y'}}", Rightarrow] \arrow[luu, "{\chi_{x,x'x'';y,y'}}"', Rightarrow] &                                                                                                                                                                                                       &    &                                                                                                                                                                                                    \\
{((\mathscr{E}_{x,y}\wedge \mathscr{E}_{x,y'})\wedge (\mathscr{E}_{x',y}\wedge \mathscr{E}_{x',y'}))\wedge (\mathscr{E}_{x'',y}\wedge \mathscr{E}_{x'',y'})} \arrow[uu] \arrow[rr]      &                  & {(\mathscr{E}_{x,yy'}\wedge \mathscr{E}_{x',yy'})\wedge \mathscr{E}_{x'',yy'}} \arrow[uu] \arrow[rr] &                                                  & {\mathscr{E}_{xx',yy'}\wedge \mathscr{E}_{x'',yy'}} \arrow[u] \arrow[lluu, "{\varphi_{x,x',x''}^{yy'}}", Rightarrow]  &                                                                                                                                                      &                                                                                                                                                                                                       &    &                                                                                                                                                                                                    \\
&                  &                                                                                                & {} \arrow[lllu, "{\chi_{x,x';y,y'}}", Rightarrow]  & {(\mathscr{E}_{xx',y}\wedge \mathscr{E}_{xx',y'})\wedge (\mathscr{E}_{x'',y} \wedge \mathscr{E}_{x'',y'})} \arrow[u] &                                                                                                                                                      & {(\mathscr{E}_{xx',y}\wedge \mathscr{E}_{x'',y})\wedge (\mathscr{E}_{xx',y'}\wedge \mathscr{E}_{x'',y'})} \arrow[ll] \arrow[luu] \arrow[ruuu, "{\varphi_{x,x',x''}^{y}\wedge \varphi_{x,x',x''}^{y'}}"', Rightarrow] &    &                                                                                                                                                                                                    \\
&                  &                                                                                                &                                                  &                                                                                                              & \circlearrowleft                                                                                                                                     &                                                                                                                                                                                                       &    &                                                                                                                                                                                                    \\
{((\mathscr{E}_{x,y}\wedge \mathscr{E}_{x',y})\wedge (\mathscr{E}_{x,y'}\wedge \mathscr{E}_{x',y'}))\wedge (\mathscr{E}_{x'',y}\wedge \mathscr{E}_{x'',y'})} \arrow[rrrruu] \arrow[uuu] &                  &                                                                                                &                                                  &                                                                                                              &                                                                                                                                                      &                                                                                                                                                                                                       &    & {((\mathscr{E}_{x,y}\wedge \mathscr{E}_{x',y})\wedge \mathscr{E}_{x'',y})\wedge (\mathscr{E}_{x,y'} \wedge \mathscr{E}_{x',y'})\wedge \mathscr{E}_{x'',y'})} \arrow[llllllll] \arrow[lluu] \arrow[uuuuuuuu]
\end{tikzcd}}

\item (4,2)-coherence axiom \label{4,2-coherence}

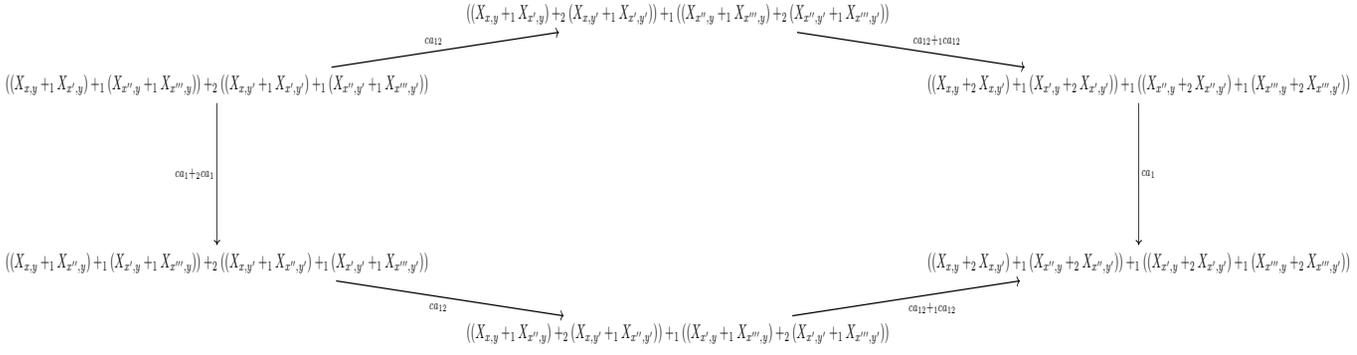
\adjustbox{scale = {0.44}{0.75},angle = -90,right = 1.7in}{%
\begin{tikzcd}
& {((X_{x,y}+_1 X_{x',y})+_2(X_{x,y'}+_1 X_{x',y'}))+_1 ((X_{x'',y}+_1 X_{x''',y})+_2(X_{x'',y'}+_1 X_{x''',y'}))} \arrow[rd, "ca_{12}+_1ca_{12}"]   &                                                                                                                                         \\
{((X_{x,y}+_1 X_{x',y})+_1(X_{x'',y}+_1 X_{x''',y}))+_2 ((X_{x,y'}+_1 X_{x',y'})+_1(X_{x'',y'}+_1 X_{x''',y'}))} \arrow[ru, "ca_{12}"] \arrow[dddd, "ca_1+_2ca_1"'] &                                                                                                                                                    & {((X_{x,y}+_2 X_{x,y'})+_1(X_{x',y}+_2 X_{x',y'}))+_1 ((X_{x'',y}+_2 X_{x'',y'})+_1(X_{x''',y}+_2 X_{x''',y'}))} \arrow[dddd, "ca_{1}"] \\
&                                                                                                                                                    &                                                                                                                                         \\
&                                                                                                                                                    &                                                                                                                                         \\
&                                                                                                                                                    &                                                                                                                                         \\
{((X_{x,y}+_1 X_{x'',y})+_1(X_{x',y}+_1 X_{x''',y}))+_2 ((X_{x,y'}+_1 X_{x'',y'})+_1(X_{x',y'}+_1 X_{x''',y'}))} \arrow[rd, "ca_{12}"']                             &                                                                                                                                                    & {((X_{x,y}+_2 X_{x,y'})+_1(X_{x'',y}+_2 X_{x'',y'}))+_1 ((X_{x',y}+_2 X_{x',y'})+_1(X_{x''',y}+_2 X_{x''',y'}))}                        \\
& {((X_{x,y}+_1 X_{x'',y})+_2 (X_{x,y'}+_1 X_{x'',y'}))+_1 ((X_{x',y}+_1 X_{x''',y})+_2(X_{x',y'}+_1 X_{x''',y'}))} \arrow[ru, "ca_{12}+_1ca_{12}"'] &                                                                                                                                        
\end{tikzcd}
}
\end{multicols}

\item \label{nodes(3,3)}The following diagrams represent categories of the partial composition laws for example, $A$ reads as 
\begin{equation*}\mathscr{E}_{a,x}\wedge \mathscr{E}_{b,x}\wedge \cdots \wedge \mathscr{E}_{c,y}\wedge \mathscr{E}_{c,z}.\end{equation*}

\begin{multicols}{4}
\adjustbox{scale = 0.5}{%
\begin{tikzcd}
z & \bullet \arrow[r, no head] & \bullet \arrow[r, no head] & \bullet                      \\
y & \bullet \arrow[r, no head] & \bullet \arrow[r, no head] & \bullet \arrow[llu, no head] \\
x & \bullet \arrow[r, no head] & \bullet \arrow[r, no head] & \bullet \arrow[llu, no head] \\
  & a                          & b                          & c                          \\
  &                            & A                       &                              
\end{tikzcd}
}
\adjustbox{scale = 0.5}{%
\begin{tikzcd}
z & \bullet \arrow[r, no head] & \bullet \arrow[rd, no head] & \bullet                      \\
y & \bullet \arrow[r, no head] & \bullet \arrow[lu, no head] & \bullet \arrow[u, no head]   \\
x & \bullet \arrow[r, no head] & \bullet \arrow[r, no head]  & \bullet \arrow[llu, no head] \\
  & a                          & b                           & c                            \\
  &                            & B_{1}                       &                             
\end{tikzcd}}

\adjustbox{scale = 0.5}{%
\begin{tikzcd}
z & \bullet \arrow[rd, no head] & \bullet \arrow[r, no head] & \bullet                      \\
y & \bullet \arrow[u, no head]  & \bullet \arrow[r, no head] & \bullet \arrow[lu, no head]  \\
x & \bullet \arrow[r, no head]  & \bullet \arrow[r, no head] & \bullet \arrow[llu, no head] \\
  & a                           & b                          & c                            \\
  &                             & B_{2}                      &                             
\end{tikzcd}
}
\adjustbox{scale = 0.5}{%
\begin{tikzcd}
z & \bullet \arrow[r, no head] & \bullet \arrow[r, no head]  & \bullet                      \\
y & \bullet \arrow[r, no head] & \bullet \arrow[rd, no head] & \bullet \arrow[llu, no head] \\
x & \bullet \arrow[r, no head] & \bullet \arrow[lu, no head] & \bullet \arrow[u, no head]   \\
  & a                          & b                           & c                            \\
  &                            & B_{3}                       &                             
\end{tikzcd}
}

\adjustbox{scale = 0.5}{%
\begin{tikzcd}
z & \bullet \arrow[r, no head]  & \bullet \arrow[r, no head] & \bullet                      \\
y & \bullet \arrow[rd, no head] & \bullet \arrow[r, no head] & \bullet \arrow[llu, no head] \\
x & \bullet \arrow[u, no head]  & \bullet \arrow[r, no head] & \bullet \arrow[lu, no head]  \\
  & a                           & b                          & c                            \\
  &                             & B_{4}                      &                             
\end{tikzcd}
}
\adjustbox{scale = 0.5}{%
\begin{tikzcd}
z & \bullet \arrow[rd, no head] & \bullet \arrow[rd, no head] & \bullet                      \\
y & \bullet \arrow[u, no head]  & \bullet \arrow[u, no head]  & \bullet \arrow[u, no head]   \\
x & \bullet \arrow[r, no head]  & \bullet \arrow[r, no head]  & \bullet \arrow[llu, no head] \\
  & a                           & b                           & c                            \\
  &                             & C_{1}                       &                             
\end{tikzcd}
}

\adjustbox{scale = 0.5}{%
\begin{tikzcd}
z & \bullet \arrow[r, no head]  & \bullet \arrow[r, no head]  & \bullet                      \\
y & \bullet \arrow[rd, no head] & \bullet \arrow[rd, no head] & \bullet \arrow[llu, no head] \\
x & \bullet \arrow[u, no head]  & \bullet \arrow[u, no head]  & \bullet \arrow[u, no head]   \\
  & a                           & b                           & c                            \\
  &                             & C_{2}                       &                             
\end{tikzcd}
}

\adjustbox{scale = 0.5}{%
\begin{tikzcd}
z & \bullet \arrow[rdd, no head] & \bullet \arrow[r, no head] & \bullet                     \\
y & \bullet \arrow[u, no head]   & \bullet \arrow[r, no head] & \bullet \arrow[lu, no head] \\
x & \bullet \arrow[u, no head]   & \bullet \arrow[r, no head] & \bullet \arrow[lu, no head] \\
  & a                            & b                          & c                           \\
  &                              & C_{5}                      &                            
\end{tikzcd}
}

\adjustbox{scale = 0.5}{%
\begin{tikzcd}
z & \bullet \arrow[r, no head] & \bullet \arrow[rdd, no head] & \bullet                    \\
y & \bullet \arrow[r, no head] & \bullet \arrow[lu, no head]  & \bullet \arrow[u, no head] \\
x & \bullet \arrow[r, no head] & \bullet \arrow[lu, no head]  & \bullet \arrow[u, no head] \\
  & a                          & b                            & c                          \\
  &                            & C_{6}                        &                           
\end{tikzcd}
}

\adjustbox{scale = 0.5}{%
\begin{tikzcd}
z & \bullet \arrow[rd, no head] & \bullet \arrow[rdd, no head] & \bullet                    \\
y & \bullet \arrow[u, no head]  & \bullet \arrow[u, no head]   & \bullet \arrow[u, no head] \\
x & \bullet \arrow[r, no head]  & \bullet \arrow[lu, no head]  & \bullet \arrow[u, no head] \\
  & a                           & b                            & c                          \\
  &                             & D_{1}                        &                           
\end{tikzcd}
}

\adjustbox{scale = 0.5}{%
\begin{tikzcd}
z & \bullet \arrow[rdd, no head] & \bullet \arrow[rd, no head] & \bullet                     \\
y & \bullet \arrow[u, no head]   & \bullet \arrow[u, no head]  & \bullet \arrow[u, no head]  \\
x & \bullet \arrow[u, no head]   & \bullet \arrow[r, no head]  & \bullet \arrow[lu, no head] \\
  & a                            & b                           & c                           \\
  &                              & D_{2}                       &                            
\end{tikzcd}
}
\adjustbox{scale = 0.5}{%
\begin{tikzcd}
z & \bullet \arrow[r, no head]  & \bullet \arrow[rdd, no head] & \bullet                    \\
y & \bullet \arrow[rd, no head] & \bullet \arrow[lu, no head]  & \bullet \arrow[u, no head] \\
x & \bullet \arrow[u, no head]  & \bullet \arrow[u, no head]   & \bullet \arrow[u, no head] \\
  & a                           & b                            & c                          \\
  &                             & D_{3}                        &                           
\end{tikzcd}
}
\end{multicols}

\begin{multicols}{2}   
\adjustbox{scale = 0.5,center}{%
\begin{tikzcd}
z & \bullet \arrow[rdd, no head] & \bullet \arrow[r, no head]  & \bullet                     \\
y & \bullet \arrow[u, no head]   & \bullet \arrow[rd, no head] & \bullet \arrow[lu, no head] \\
x & \bullet \arrow[u, no head]   & \bullet \arrow[u, no head]  & \bullet \arrow[u, no head]  \\
  & a                            & b                           & c                           \\
  &                              & D_{4}                       &                            
\end{tikzcd}
}
\adjustbox{scale = 0.5,center}{%
\begin{tikzcd}
z & \bullet \arrow[rdd, no head] & \bullet \arrow[rdd, no head] & \bullet                    \\
y & \bullet \arrow[u, no head]   & \bullet \arrow[u, no head]   & \bullet \arrow[u, no head] \\
x & \bullet \arrow[u, no head]   & \bullet \arrow[u, no head]   & \bullet \arrow[u, no head] \\
  & a                            & b                            & c                          \\
  &                              & E                            &                           
\end{tikzcd}
}
\end{multicols}
\item Interaction between two $\mu$'s via $\chi$.\label{diagram3}

\adjustbox{scale = 0.75, angle=-90, center}{%
\begin{tikzcd}
&                                                                                                                               &                                                                                                                                                                                                                                                                           & {\mathscr{E}_{xx',yy'}}                                   &                                                                                                                                                                                                   &                                                                                                                              &                                                  \\
& {}                                                                                                                            & {\mathscr{E}_{xx',y'}\wedge \mathscr{E}_{xx',y}} \arrow[ru]                                                                                                                                                                                                                  & {}                                                      & {\mathscr{E}_{x',yy'}\wedge \mathscr{E}_{x,yy'}} \arrow[lu]                                                                                                                                          & {}                                                                                                                           &                                                  \\
&                                                                                                                               & {(\mathscr{E}_{x',y'}\wedge \mathscr{E}_{x,y'})\wedge (\mathscr{E}_{x',y}\wedge \mathscr{E}_{x,y})} \arrow[rr, "{c_{x',x;y',y}}"] \arrow[u] \arrow[ru, "{\chi_{x',x;y',y}}", Rightarrow, xshift = 28, yshift = 10] \arrow[lu, "{\mu_{xx';y,y'}}", Rightarrow]                                                 &                                                         & {(\mathscr{E}_{x',y'}\wedge \mathscr{E}_{x',y})\wedge (\mathscr{E}_{x,y'}\wedge \mathscr{E}_{x,y})} \arrow[u] \arrow[ru, "{\mu_{x,x';yy'}}"', Rightarrow]                                               &                                                                                                                              &                                                  \\
{\mathscr{E}_{xx',yy'}} \arrow[rrrddd, bend right,"id"'] \arrow[rrruuu, bend left,"id"] & {(\mathscr{E}_{x',y}\wedge \mathscr{E}_{x,y})\wedge (\mathscr{E}_{x',y'}\wedge \mathscr{E}_{x,y'})} \arrow[ru, "\tau"'] \arrow[l] &                                                                                                                                                                                                                                                                           & {} \arrow[loop, distance=2em, in=305, out=235, yshift = 20]          &                                                                                                                                                                                                   & {(\mathscr{E}_{x,y'}\wedge \mathscr{E}_{x,y})\wedge (\mathscr{E}_{x',y'}\wedge \mathscr{E}_{x',y})} \arrow[lu, "\tau"] \arrow[r] & {\mathscr{E}_{xx',yy'}} \arrow[llluuu, bend right,"id"'] \\
&                                                                                                                               & {(\mathscr{E}_{x,y}\wedge \mathscr{E}_{x',y})\wedge (\mathscr{E}_{x,y'}\wedge \mathscr{E}_{x',y'})} \arrow[rr, "{c_{x,x';y,y'}}"] \arrow[lu, "\tau\wedge\tau"'] \arrow[d] \arrow[rd, "{\chi_{x,x';y,y'}}", Rightarrow, xshift = 28, yshift = -10] \arrow[ld, "{\mu_{x,x';y}\wedge \mu_{x,x';y'}}"', Rightarrow] &                                                         & {(\mathscr{E}_{x,y}\wedge \mathscr{E}_{x,y'})\wedge (\mathscr{E}_{x',y}\wedge \mathscr{E}_{x',y'})} \arrow[ru, "\tau\wedge \tau"] \arrow[d] \arrow[rd, "{\mu_{x;y,y'}\wedge \mu_{x';y,y'}}", Rightarrow] &                                                                                                                              &                                                  \\
    & {}                                                                                                                            & {\mathscr{E}_{xx',y}\wedge \mathscr{E}_{xx',y'}} \arrow[rd]                                                                                                                                                                                                                  & {}                                                      & {\mathscr{E}_{x,yy'}\wedge \mathscr{E}_{x',yy'}} \arrow[ld]                                                                                                                                          & {}                                                                                                                           &                                                  \\
&                                                                                                                               &                                                                                                                                                                                                                                                                           & {\mathscr{E}_{xx',yy'}} \arrow[rrruuu, "id"', bend right] &                                                                                                                                                                                                   &                                                                                                                              &                                                 
\end{tikzcd}
}

\end{enumerate}

\bibliographystyle{hamsalpha}
\bibliography{references}
\end{document}